\documentclass[12pt]{article}

\usepackage{mathrsfs}
\usepackage{amsmath}
\usepackage{amsthm}
\usepackage{amssymb}
\usepackage{amsfonts}
\usepackage{epsf}
\usepackage{graphicx}
\usepackage{caption2}
\usepackage{comment}
\usepackage{overpic}

\renewcommand{\l}{\lambda}
\newcommand{\D}{\displaystyle}

\newtheorem{thm}{Theorem}[section]
\newtheorem{lem}[thm]{Lemma}

\newtheorem{prop}[thm]{Proposition}
\theoremstyle{definition}
\newtheorem{defn}[thm]{Definition}
\theoremstyle{remark}

\DeclareMathOperator*\Res{{Res}}   \DeclareMathOperator\re{{Re}}
\DeclareMathOperator\im{{Im}}

\numberwithin{equation}{section}

\begin{document}

\title{ Painlev\'{e} IV asymptotics for orthogonal polynomials with respect
to a modified Laguerre weight}
\date{\today}
\author{D. Dai and A.B.J. Kuijlaars$^*$ \\
\small{Department of Mathematics, Katholieke Universiteit Leuven,} \\
\small{Celestijnenlaan 200 B, 3001 Leuven, Belgium.} \\
\small{dan.dai@wis.kuleuven.be, arno.kuijlaars@wis.kuleuven.be}}

\maketitle


\begin{abstract}

We study polynomials that are orthogonal with respect to the modified Laguerre weight $z^{-n + \nu} e^{-Nz} (z-1)^{2b}$ in the
limit where $n, N \to \infty$ with $N/n \to 1$ and $\nu$ is a fixed number in $\mathbb{R} \setminus \mathbb{N}_0$. With the
effect of the factor $(z-1)^{2b}$, the local parametrix near the critical point $z =1$ can be constructed in terms of
$\Psi$-functions associated with the Painlev\'e IV equation. We show that the asymptotics of the recurrence coefficients of
orthogonal polynomials can be described in terms of specified solution of the Painlev\'e IV equation in the double scaling
limit. Our method is based on the Deift/Zhou steepest decent analysis of the Riemann-Hilbert problem associated with orthogonal
polynomials.

\end{abstract}

\vspace{30mm}

\hrule width 65mm

\vspace{5mm}


$^*$The authors are partially supported by K.U.Leuven research grant OT/04/21 and by the Belgian Interuniversity Attraction Pole
P06/02. In addition, the second author is partially supported by FWO-Flanders project G.0455.04, by the European Science
Foundation Program MISGAM, and by the Ministry of Education and Science of Spain, project code MTM2005-08648-C02-01.


\section{Introduction and statement of results}

\subsection{Modified Laguerre weights}

Let $\alpha$, $b$ be real constants, and let $N$ be a positive number.  Define the (complex-valued) weight function to be
\begin{equation} \label{weight}
w(z) = z^{\alpha} e^{-Nz} (z-1)^{2b},
\end{equation}
where $z^{\alpha}$ and $(z-1)^{2b}$ are defined with cuts along $[0,\infty)$ and $[1,\infty)$, respectively. That is,
$z^{\alpha} = |z|^{\alpha} e^{i \alpha \arg z}$ with $0 < \arg z < 2\pi$, and $(z-1)^{2b} = |z-1|^{2b} e^{2 i b \arg (z-1)}$
with $0 < \arg(z-1) < 2\pi$. Let $\Sigma$ be a contour in $\mathbb C \setminus [0,\infty)$ that is symmetric with respect to the
real axis, tends to infinity in the horizontal direction and remains bounded in the vertical direction. This contour divides the
complex plane into two domains $\Omega_{\pm}$; see Figure~\ref{Sigmacontour}. We consider monic polynomials $\pi_n$ of degree
$n$, $\pi_n(z) = z^n + \cdots$, that are orthogonal with respect to $w(z)$ on $\Sigma$, i.e.,
\begin{equation} \label{orthogonality}
    \int_{\Sigma} \pi_n(z) z^j w(z) dz = 0, \qquad j=0, 1, \ldots, n-1.
    \end{equation}
This is an example of non-Hermitian orthogonality, for which there is no general existence and uniqueness result associated with
such $w(z)$. It will be part of our results that in the asymptotic regime that we will consider, the polynomial $\pi_n$ uniquely
exists for $n$ large enough.

\begin{figure}[h]
\centering
\includegraphics[width=200pt]{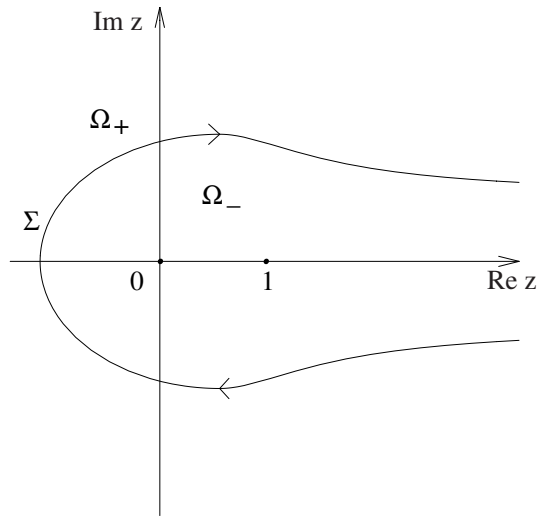}
\caption{The contour $\Sigma$} \label{Sigmacontour}
\end{figure}

For $b=0$ and $\alpha > -1$ (with $\alpha \not\in \mathbb Z$) we may  deform the contour of integration to $[0,\infty)$. Then
the orthogonality condition (\ref{orthogonality}) reduces to
\begin{equation} \label{orthogonality2}
    \int_0^{\infty} \pi_n(x) x^{j+\alpha} e^{-Nx} dx = 0, \qquad j=0, 1, \ldots, n-1,
    \end{equation}
which shows that in this case $\pi_n$ is related to the classical Laguerre polynomial with parameter $\alpha$,
\begin{equation} \label{classicalLaguerre}
    \pi_n(z) = (-1)^n \frac{n!}{N^n} L_n^{(\alpha)}(Nz).
    \end{equation}
For properties of the classical Laguerre polynomial $L_n^{(\alpha)}(z)$, see e.g.\ Abramowitz and Stegun \cite{as} or Szeg\H{o}
\cite{s}. For $b =0$ and $\alpha < -1$, the orthogonality condition (\ref{orthogonality2}) is not valid, but the relation
(\ref{classicalLaguerre}) still holds true. Asymptotic properties of the Laguerre polynomials with large negative parameters
were studied by Kuijlaars and McLaughlin \cite{km1,km2} in the regime where $N = n$, $n \to \infty$, $\alpha \to -\infty$ in
such a way that
\[ \lim_{n \to \infty}  - \frac{\alpha}{n} = A \]
exists. Then it was found that the zeros of $\pi_n$ accumulate either on an open contour (if $A > 1$) or on a union of a closed
contour and a real interval (if $0 < A < 1$); see Figure~\ref{lagzero}. For the special value $A=1$, the zeros typically (but
not always!) accumulate on the Szeg\H{o} curve $\mathcal{S}$ defined by
\begin{equation} \label{szego-curve}
    \mathcal{S} := \{ z \in \mathbb{C} : |z e^{1-z}| = 1 \quad \textrm{and} \quad |z| \leq 1 \}.
\end{equation}
In the special case that $\alpha = -n + \nu$ with $\nu \notin \mathbb{N}_0 = \mathbb{N} \cup \{0\}$ fixed, it was also found
\cite{dev,kmm} that the local asymptotics of the polynomials $\pi_n$ of (\ref{classicalLaguerre}) near the point $z=1$ are given
in terms of parabolic cylinder functions $D_{\nu}(z)$. The main tool in the analysis of \cite{km1,km2,dev,kmm} is the matrix
valued Riemann-Hilbert (RH) problem for Laguerre polynomials and the Deift/Zhou method of steepest descent, which was introduced
in \cite{dz} and first applied to orthogonal polynomials in \cite{dkmvz,dkmvz2}; see also \cite{deift}. The parabolic cylinder
functions appear via a construction of a local parametrix in a neighborhood of the point $z=1$.

\begin{figure}[h]
\centering
\includegraphics[width=400pt]{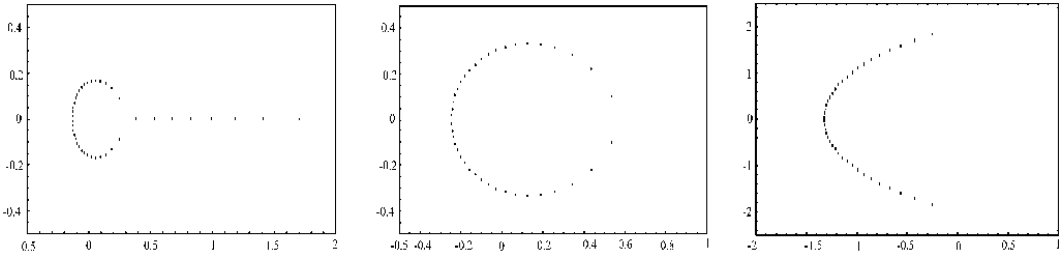}
\caption{Zeros of $L_n^{(-nA)}(nz)$ for $n=40$, and $A=0.81$ (left), $A=1.001$ (middle), and $A=2$ (right).} \label{lagzero}
\end{figure}

The main effect of the extra factor $(z-1)^{2b}$ in the weight (\ref{weight}) is in the local behavior of the polynomials near
the point $z=1$. From the point of view of Riemann-Hilbert analysis it means that in the critical case $A=1$ the construction of
the local parametrix with parabolic cylinder functions will no longer work. It is the aim of this paper to show that in this
situation the role of the parabolic cylinder functions is replaced by the $\Psi$ functions associated with a special solution of
the Painlev\'e IV equation. To the best of our knowledge, this is the first time that Painlev\'e IV functions are used in the
construction of a local parametrix in a RH steepest descent analysis.

We were guided in our approach by a number of recent advances in random matrix theory and the theory of orthogonal polynomials.
In critical situations it was found that Painlev\'e transcendents appear in the description of the local eigenvalue statistics
of random matrices as well as in the local asymptotics of orthogonal polynomials near singular points. For example, the local
eigenvalue statistics at the opening of a gap in the spectrum of unitary invariant random matrix ensembles is described by the
Hastings-McLeod solution of the Painlev\'e II equation \cite{bi,ck,ckv}. The same function  is well-known to play a role in the
Tracy-Widom distributions for the largest eigenvalue of random matrices \cite{tw} and the problem of the length of the longest
increasing subsequence of a random permutation \cite{bdj}. Singular behavior at edge points may be described by special
solutions of the Painlev\'e I equation \cite{dk,fik}, the second member of the Painlev\'e I hierarchy \cite{cg,cv}, and the
Painlev\'e XXXIV equation \cite{ikj}.  The Riemann-Hilbert approach was used in all of these situations.

In the statement of the results in this paper we will focus on the behavior of the recurrence coefficients in the three-term
recurrence relation
\begin{equation} \label{3term}
    \pi_{n+1}(z) = (z-b_n) \pi_{n}(z) - a_n \pi_{n-1}(z)
\end{equation}
in the asymptotic regime where $n \to \infty$ with
\begin{equation} \label{ta-def}
    \alpha = -n + \nu, \qquad N = n + \sqrt{2} \, L n^{1/2}
\end{equation}
with $b$, $\nu$ and $L$ are real constants and independent of $n$. Then of course the recurrence coefficients depend on the
constants $b$, $\nu$ and $L$, but we suppress it in our notation. In this regime it turns out that we see the appearance of the
Painlev\'e IV transcendent.

As noted before, for $b=0$ the polynomials reduce to rescaled Laguerre polynomials (\ref{classicalLaguerre}). From the
recurrence relation for Laguerre polynomials
\begin{equation} \label{Laguerrerecurrence1}
    - z L_n^{(\alpha)}(z) =
    (n+1) L_{n+1}^{(\alpha)}(z) - (2n + \alpha + 1) L_n^{(\alpha)}(z) + (n + \alpha) L_{n-1}^{(\alpha)}(z)
\end{equation}
and (\ref{classicalLaguerre}) it follows that under the scaling (\ref{ta-def}) the recurrence coefficients behave as
\begin{equation} \label{Laguerrerecurrence2}
    a_n = \frac{\nu}{n} + O(n^{-3/2}), \qquad
    b_n = 1 - \frac{\sqrt{2} \, L}{n^{1/2}} + O(n^{-1}), \qquad \text{as } n \to \infty.
\end{equation}

For general $b \in \mathbb{R}$ and $\nu \not\in \mathbb N_0$, we prove that  (except for a discrete set of values $L$) $a_n$ and
$b_n$ exist for $n$ large enough, and satisfy
\[ a_n = \frac{a^{(1)}}{n} + O(n^{-3/2}), \qquad
   b_n = 1 - \frac{b^{(1)}}{n^{1/2}} + O(n^{-1}), \]
for certain explicit constants $a^{(1)}$ and $b^{(1)}$, which depend on $b$, $\nu$ and $L$, that are explicitly calculated in
terms of functions associated with the Painlev\'e IV equation. For $b=0$ they reduce to $a^{(1)} = \nu$ and $b^{(1)} = \sqrt{2}
\, L$.

Before we can state our result we first discuss the Painlev\'e IV equation.


\subsection{Painlev\'e IV equation} \label{rh-piv}

The Painlev\'{e} IV (PIV) equation is defined as
\begin{equation} \label{PIV}
    \frac{d^2 u}{d s^2} =  \frac{1}{2u} \left( \frac{du}{ds} \right)^2 + \frac{3}{2} u^3 +
    4 s u^2 + 2(s^2 + 1 - 2 \Theta_\infty)u - \frac{8 \Theta^2}{u},
\end{equation}
where $\Theta$ and $\Theta_\infty$ are constants.  It is known that all solutions $u(s)$ of PIV are meromorphic functions in the
complex $s$-plane and all poles of $u(s)$ are simple with the residue $\pm 1$; see \cite{gls,luka,stein}. For more properties of
the PIV transcendent, see a very good review article by Clarkson \cite{clarkson} and references therein. The $\Psi$-functions
associated with Painlev\'e IV are solutions of the following system of linear differential equations (Lax pair) for $\Psi(\l,s)$
given first by Jimbo and Miwa \cite{jm}
\begin{equation} \label{PIVsystem}
\frac{\partial \Psi}{\partial \l} = A \Psi, \quad  \frac{\partial \Psi}{\partial s} = B \Psi,
\end{equation}
with
\begin{align}
A = & \left( \l + s + \frac{1}{\l}(\Theta - K) \right) \sigma_3 +
    y \left( 1 - \frac{u}{2 \l} \right) \sigma_+ \nonumber \\
   & \qquad + \frac{2}{y} \left( K - \Theta - \Theta_\infty + \frac{K}{ \l u} (K - 2 \Theta) \right) \sigma_-, \\
B  = & \l \sigma_3 + y \sigma_+ + \frac{2}{y}(K -\Theta - \Theta_\infty) \sigma_-,
\end{align}
where $y = y(s)$ and $K = K(s)$ are defined by
\begin{align}
    \frac{1}{y} \frac{dy}{d s} & =  - u - 2s,
    \label{y-def} \\
    K & =  \frac{1}{4} \Big( - \frac{du}{d s} + u^2 + 2s \, u + 4 \Theta \Big)
\label{k-def}
\end{align}
and
\begin{equation} \label{sigma-def}
    \sigma_3 = \left( \begin{matrix} 1 & 0 \\ 0 & -1 \end{matrix} \right),  \quad
    \sigma_+ = \left( \begin{matrix} 0 & 1 \\ 0 & 0 \end{matrix} \right), \quad
    \sigma_- = \left( \begin{matrix} 0 & 0 \\ 1 & 0 \end{matrix} \right).
\end{equation}
The Painlev\'e IV equation (\ref{PIV}) is the compatibility condition for the system (\ref{PIVsystem}). It should be noted that
the Lax pair is not unique. For example, one can find another Lax pair in Kitaev \cite{kitaev} and Milne, Clarkson and Bassom
\cite{mcb}.

To explain the special solutions of PIV that will appear in our results, we recall the RH problem for $\Psi(\l,s)$ associated
with PIV.
\begin{figure}[htb]
\centering \begin{overpic}[scale=0.8]%
{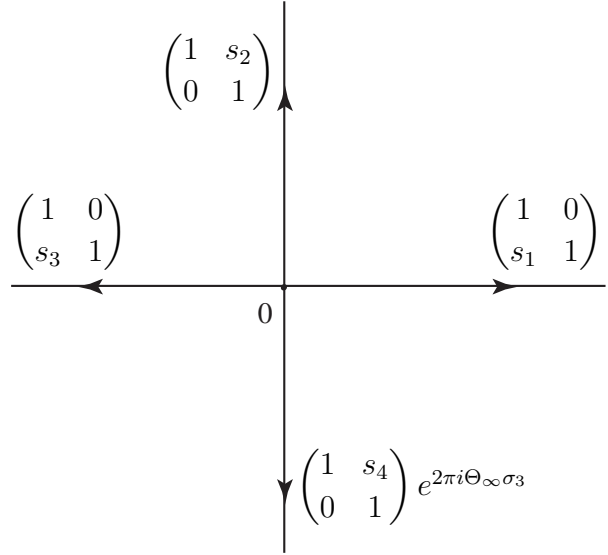}%
\put(80,53){$\begin{pmatrix} 1 & 0 \\ s_1 & 1 \end{pmatrix}$} \put(25,80){$\begin{pmatrix} 1 & s_2 \\ 0 & 1 \end{pmatrix}$}
\put(0,53){$\begin{pmatrix} 1 & 0
\\ s_3 & 1 \end{pmatrix}$} \put(48,10){$\begin{pmatrix} 1 & s_4 \\ 0 & 1 \end{pmatrix} e^{2\pi i \Theta_{\infty} \sigma_3}$}
\end{overpic}
\caption{The contour and the jump matrices for the RH problem for $\Psi$} \label{psicontour}
\end{figure}
In the RH problem $\Psi(\lambda,s)$ is viewed as a function of $\lambda$, with $s$ appearing as a parameter in the asymptotic
condition (\ref{psi-asy}). The constants $\Theta_{\infty}$ and $\Theta$ appear in the behaviors at infinity (\ref{psi-asy}) and
at the origin (\ref{psi-0}), respectively. The jump matrices in the RH problem include four parameters (Stokes multipliers)
$s_j$, $j=1,\ldots, 4$, satisfying
\begin{equation} \label{stokes}
    (1 + s_2 s_3) e^{2 \pi i \, \Theta_\infty} + [s_1 s_4 + (1 + s_3 s_4)(1 + s_1 s_2)]
    e^{-2 \pi i \, \Theta_\infty} = 2 \cos{2 \pi \Theta}.
\end{equation}
The RH problem for $\Psi$ is (see \cite[Section 5.1]{fikn} and \cite{fma})
\begin{enumerate}
\item[(a)] $\Psi(\l, s)$ is analytic for $ \l \in \mathbb{C} \setminus
(\mathbb{R} \cup i \, \mathbb{R})$; see Figure~\ref{psicontour};
\item[(b)] on the contours $\mathbb{R} \cup i \, \mathbb{R}$ (with
orientation as in Figure~\ref{psicontour}),
\begin{align}
    \Psi_+(\l,s) & =  \Psi_-(\l,s) \, S_1, \hspace{55pt} \textrm{for } \l \in \mathbb{R}_+; \label{psi-rhb} \\
    \Psi_+(\l,s) & =  \Psi_-(\l,s) \, S_2, \hspace{55pt} \textrm{for } \l \in i \, \mathbb{R}_+; \\
    \Psi_+(\l,s) & =  \Psi_-(\l,s) \, S_3, \hspace{55pt} \textrm{for } \l \in \mathbb{R}_-; \\
    \Psi_+(\l,s) & =  \Psi_-(\l,s) \, S_4 \, e^{2 \pi i \Theta_\infty \sigma_3}, \quad \textrm{for } \l \in i \, \mathbb{R}_-,
\end{align}
where
\begin{equation} \label{psi-sj}
    \begin{array}{ll}
        S_1 = \left( \begin{matrix} 1 & 0 \\ s_1 & 1 \end{matrix} \right), & \qquad
        S_2 = \left( \begin{matrix} 1 & s_2 \\ 0 & 1 \end{matrix} \right), \\
        S_3 = \left( \begin{matrix} 1 & 0 \\ s_3 & 1 \end{matrix} \right), & \qquad
        S_4 = \left( \begin{matrix} 1 & s_4 \\ 0 & 1 \end{matrix} \right);
\end{array}
\end{equation}
\item[(c)] as $\l \to \infty$:
\begin{equation} \label{psi-asy}
    \Psi(\l,s) = \left( I + \frac{\Psi_{-1}(s)}{\l} + \frac{\Psi_{-2}(s)}{\l^2} +
    O\left( \frac{1}{\l^3} \right) \right) e^{(\frac{\l^2}{2} + s \l) \sigma_3} \
    \l^{-\Theta_\infty \sigma_3},
\end{equation}
where $\l^{-\Theta_\infty}$ is defined with a branch cut along $i \, \mathbb{R}_-$;

\item[(d)] as $\l \to 0$:
\begin{equation} \label{psi-0}
    \Psi(\l,s) \mathcal{C}^{-1} \l^{-\Theta \sigma_3} = O(1)  \qquad \textrm{for } \re \l > 0 \ \textrm{and} \ \im \l < 0 .
\end{equation}
where $\l^{\Theta}$ is defined with a branch cut along $i \, \mathbb{R}_-$ and the connection matrix $\mathcal{C}$ is any
invertible matrix satisfying
\begin{equation} \label{c-def}
    S_1 S_2 S_3 S_4 e^{2 \pi i \Theta_\infty \sigma_3} = \mathcal{C}^{-1} e^{-2 \pi i \Theta \sigma_3} \mathcal{C}.
\end{equation}
\end{enumerate}

Then, from Theorem 5.2 in \cite{fikn}, we obtain a  solution of PIV (\ref{PIV}) with constants $\Theta$ and $\Theta_{\infty}$,
by putting
\begin{equation} \label{pivu}
    u(s) :=  -2s - \frac{d}{ds} \log \Big((\Psi_{-1})(s)\Big)_{12}.
\end{equation}
The solution depends on the constants $s_j$'s in (\ref{psi-sj}), which as explained should satisfy the equation (\ref{stokes}).
The equation (\ref{stokes}) is invariant under the transformation
\begin{equation} \label{sm-trans}
    \{ s_1, s_2, s_3, s_4 \} \mapsto \{ d s_1, d^{-1} s_2, d s_3, d^{-1} s_4 \},
    \qquad d \neq 0.
\end{equation}
The $\Psi$ function then transforms as $\Psi \mapsto d^{-\sigma_3/2} \Psi d^{\sigma_3/2}$ which leaves the solution (\ref{pivu})
invariant under the transformation (\ref{sm-trans}). The connection matrix $\mathcal{C}$ transforms as $\mathcal{C} \mapsto
d^{-\sigma_3/2} \mathcal{C} \, d^{\sigma_3/2}$. The PIV solution (\ref{pivu}) is a meromorphic function of $s$ and the RH
problem for $\Psi$ has a solution if and only if $s$ is not a pole of $u$.

The special solution of (\ref{PIV}) that will appear in our results corresponds to the special Stokes multipliers
\begin{equation} \label{case0-sm}
\begin{array}{ll}
    s_1 = (e^{ (\Theta_{\infty} - \Theta ) \pi i} -e^{-(\Theta_{\infty}-\Theta) \pi i}) e^{\Theta_{\infty} \pi i}, &
    s_2 = e^{-\Theta \pi i}, \\
    s_3 = -(e^{(\Theta_{\infty} + \Theta) \pi i} - e^{-(\Theta_{\infty}+\Theta) \pi i}) e^{-\Theta_{\infty}\pi i}, &
    s_4 = -e^{(2 \Theta_{\infty} + \Theta)\pi i}.
\end{array}
\end{equation}
where
\begin{equation}
\Theta = -b \quad \textrm{and} \quad \Theta_{\infty} = \nu + b.
\end{equation}
It is easy to check that (\ref{case0-sm}) indeed satisfies the relation (\ref{stokes}). Furthermore, by (\ref{c-def}), the
corresponding connection matrix $\mathcal{C}$ is
\begin{equation*}
\mathcal{C} = \left( \begin{matrix} r_1 & 0 \\ 0 & r_1^{-1} \end{matrix} \right)  \left( \begin{matrix} 1 & 0 \\ - e^{(2
\Theta_\infty - \Theta) \pi i} & 1
\end{matrix} \right) \left( \begin{matrix} 1 & r_2 \\ 0 & 1 \end{matrix} \right), \quad r_1 \neq 0
\end{equation*}
and $r_1$ and $r_2$ are arbitrary constants satisfying $r_2 = 0, \hbox{if } b = m/2, m \in \mathbb{Z}. $ Throughout this paper,
we will choose $r_1 \equiv 1$ and $r_2 \equiv 0$, which means
\begin{equation} \label{c-explicit}
\mathcal{C} =  \left( \begin{matrix} 1 & 0 \\ - e^{(2 \Theta_\infty - \Theta) \pi i} & 1
\end{matrix} \right),
\end{equation}
or after a transformation (\ref{sm-trans})
\begin{equation} \label{c-explicit-2}
\mathcal{C} =  \left( \begin{matrix} 1 & 0 \\ - d e^{(2 \Theta_\infty - \Theta) \pi i} & 1
\end{matrix} \right).
\end{equation}


\subsection{Statement of results}

The asymptotic formulas for the recurrence coefficients $a_n$ and $b_n$ in (\ref{3term}) are given as follows:
\begin{thm} \label{thm1}
Suppose $b \in \mathbb R$, $\nu \in \mathbb R \setminus \mathbb N_0$ and put
\begin{equation}
    \Theta=-b, \qquad \Theta_\infty = \nu + b .
\end{equation}
Let $u(s)$ be the special solution of PIV corresponding to the Stokes multipliers \eqref{case0-sm} and let $K(s)$ be defined in
(\ref{k-def}). Assume
\begin{equation} \label{alpha-N-def}
    \alpha = - n + \nu, \qquad N = n + \sqrt{2} \, L n^{1/2}
\end{equation}
where $L$ is not a pole of $u(s)$. Then, for every large enough $n$, the monic polynomial $\pi_n$ of degree $n$, satisfying
\[ \int_{\Sigma} \pi_n(z) z^{j-\alpha} e^{-Nz} (z-1)^{2b} dz = 0, \qquad j=0,1, \ldots, n-1, \]
uniquely exists. The polynomials satisfy the recurrence relation \eqref{3term} with recurrence coefficients $a_n$ and $b_n$
satisfying
\begin{equation} \label{an-final}
    a_n = \frac{1}{n} \Big( \nu - K(L) \Big)+ O(n^{-3/2}) \qquad \hbox{as } n \to \infty
\end{equation}
and, if $K(L) \neq \nu$,
\begin{equation} \label{bn-final}
    b_n = 1 + \frac{\sqrt{2}}{\sqrt{n}} \left[ \frac{K(L) \Big( K(L) + 2b \Big)}{u(L) \Big( K(L) - \nu \Big)} - L
\right] + O(n^{-1}) \qquad \hbox{as } n \to \infty.
\end{equation}
\end{thm}

In the Theorem \ref{thm1}, (\ref{bn-final}) still makes sense when $u(L) = 0$. In fact, from the PIV equation in (\ref{PIV}), it
can be verified that
$$u'(L) = \pm 4 \Theta = \mp 4b \qquad \textrm{if} \quad u(L) = 0.$$
From (\ref{k-def}), then we have $K(L) = 0$ or $K(L) = 2 \Theta = -2b$. Then by L'Hospital's rule, we get from (\ref{bn-final})
\begin{eqnarray}
b_n & = 1 + \frac{\sqrt{2}}{\sqrt{n}} \left[ \frac{K'(L)}{2 \nu} - L \right] + O(n^{-1}), \qquad \hbox{if }
u(L)= 0 \hbox{ and } u'(L) = -4b, \\
b_n & = 1 + \frac{\sqrt{2}}{\sqrt{n}} \left[ \frac{K'(L)}{2 (\nu + 2b)} - L \right] + O(n^{-1}), \qquad \hbox{if } u(L)= 0
\hbox{ and } u'(L) = 4b. \label{bn-final-2}
\end{eqnarray}
Note that the factor $(\nu + 2b)$ in (\ref{bn-final-2}) can not be zero because of the assumption $K(L) \neq \nu$ in
(\ref{bn-final}).

When $L$ is not a pole of $u(s)$, it is still possible that $K(L)=\nu$. In this case, the RH problem for $\Psi$ in
(\ref{psi-rhb})--(\ref{psi-0}) is solvable. Moreover, for large enough $n$, the monic polynomial $\pi_n$ still uniquely exists.
However something strange occurs for the recurrence coefficient $b_n$. Note that (\ref{bn-final}) no longer holds when
$K(L)=\nu$. In fact, as $n \to \infty$, $b_n$ does not tend to 1 but to another constant, which means there is a sudden change
from $K(L) \neq \nu$ to $K(L)=\nu$. At this time, we can not explain why this phenomenon happens.

In \cite{murata}, Murata found the following Schlesinger transformation
\begin{equation} \label{schlesinger}
u^*(s) = - \frac{ 2 \, K(s) \Big( K(s) + 2b \Big)}{u(s) \Big( K(s) - \nu \Big)},
\end{equation}
where $u^*(s)$ is a solution to the PIV equation with parameters $\Theta = -b, \Theta_\infty = \nu + b + 1$. Then, with
(\ref{schlesinger}), one can rewrite (\ref{bn-final}) as
\begin{equation}
 b_n = 1 - \frac{\sqrt{2}}{\sqrt{n}} \biggl( \frac{1}{2} \, u^*(L) + L
\biggr) + O(n^{-1}), \qquad n \to \infty,
\end{equation}
whenever $L$ is not a pole of $u^*(s)$.

The case $\nu \in \mathbb N_0$ is special and we do not understand it at this moment. That this is a particular situation can
already be seen in the Laguerre case treated in \cite{km2}, since it corresponds to Laguerre polynomials with a zero of large
order at the origin, and the RH problem stated in Section \ref{rh-op} below does not have a solution in case $\nu \in \mathbb
N_0$ and $n
> \nu$, see Proposition 2.2 of \cite{km1}.

Theorem \ref{thm1} could still be true in case $\nu \in \mathbb N_0$, but our method of proof fails. Indeed, in the
transformation (\ref{t-y}) we make essential use of the fact that $\nu$ is not an integer. Therefore we exclude the case $\nu
\in \mathbb N_0$ from our considerations.

Using the Riemann-Hilbert method, it is possible to obtain asymptotic expansion for polynomials $\pi_n(z)$ in all regions of the
full complex $z$-plane as $n \rightarrow \infty$. Near the point $z = 1$, the expansion involves the $\Psi$-functions given in
(\ref{PIVsystem}). We will not discuss it in this paper.

For the polynomials, we restrict ourselves to the asymptotic zero distribution.

\begin{thm} \label{thm3}
Under the same assumptions as in Theorem \ref{thm1} and $K(L) \neq \nu$, we have that all the zeros of polynomials $\pi_{n}(z)$
tend to the Szeg\H{o} curve $\mathcal{S}$ given in \eqref{szego-curve}. More precisely, for any neighborhood $U(\mathcal{S})$ of
$\mathcal{S}$, there exists a positive integer $n_0$ such that for any $n > n_0$, all zeros of $\pi_n(z)$ lie in
$U(\mathcal{S})$.
\end{thm}


\subsubsection*{Outline of the rest of the paper}

In Section \ref{rh-op}, we formulate the RH problem for orthogonal polynomials $\pi_n$, which is the starting point of the
asymptotic analysis.  We have to distinguish the two cases
\begin{description}
\item[Case I:] $\nu > 0$ and $\nu \not\in \mathbb N$,
\item[Case II:] $\nu < 0$,
\end{description}
since the steepest descent analysis is different for the two cases. The two cases are dealt with in Sections \ref{case1} and
\ref{case2} respectively. Together they contain the proof of Theorem \ref{thm1}. In Section \ref{thm3-proof}, we prove Theorem
\ref{thm3}. Finally in Section \ref{special}, we show that, with certain parameters $\Theta$ and $\Theta_\infty$, the special
solutions of PIV studied in this paper are given in terms of parabolic cylinder functions.

Since the two cases Case I and Case II are independent of each other, we are going to use the same notation for the functions
and variables. We trust that this will not lead to any confusion.


\section{The RH problem for orthogonal polynomials} \label{rh-op}

Consider the following RH problem for a $2 \times 2$ matrix valued function $Y: \mathbb{C} \setminus \Sigma \mapsto
\mathbb{C}^{2 \times 2}$:

\begin{enumerate}
\item[(a)] $Y(z)$ is analytic for $z \in \mathbb{C} \setminus
\Sigma$; see Figure~\ref{Sigmacontour}. Here $\Sigma$ is an oriented curve whose $+$-side and $-$-side are on the left and right
while traversing the contour, respectively, \item[(b)] $Y(z)$ possesses continuous boundary values $Y_{\pm}(z)$ such that
\begin{equation}
Y_+(z) = Y_-(z) \left( \begin{matrix} 1 & w(z) \\ 0 & 1 \end{matrix} \right) \quad \textrm{for } z \in \Sigma.
\end{equation}
\item[(c)] $Y(z)$ has the following behavior as $z \to \infty$
\begin{equation*}
Y(z)=\left( I + O\left(\frac{1}{z}\right) \right) \left( {\begin{array}{cc} z^n & 0 \\
0 & z^{-n} \end{array}} \right).
\end{equation*}
\end{enumerate}

By a by now standard argument, originally due to Fokas, Its, and Kitaev \cite{fik}, the solution of the above RH problem, if it
exists, is uniquely given by
\begin{equation} \label{y-sol}
    Y(z) = \begin{pmatrix} \pi_n(z) & \displaystyle \frac{1}{2 \pi i}
        \int_\Sigma \frac{w(s) \, \pi_n(s)}{s-z} ds \\
        p_{n-1}(z) & \displaystyle \frac{1}{2 \pi i} \int_{\Sigma} \frac{ w(s) \, p_{n-1}(s)}{s-z} ds
        \end{pmatrix},
\end{equation}
where $\pi_n$ is a monic polynomial of degree $n$ satisfying (\ref{orthogonality}) and $p_{n-1}$ is a polynomial of degree $\leq
n-1$. The existence and uniqueness of the monic polynomial $\pi_n$ satisfying (\ref{orthogonality}) is equivalent to the
solvability of the RH problem.

The recurrence coefficients $a_n$ and $b_n$ can be expressed in terms of the solution of the RH problem.

\begin{prop} 
Assume that the RH problem has a solution $Y$ and write
\begin{equation} \label{Y-expansion}
    Y(z) = \left( I + \frac{Y_{-1}}{z} + \frac{Y_{-2}}{z^2} + O(z^{-3}) \right)
    \begin{pmatrix} z^n & 0 \\ 0 & z^{-n} \end{pmatrix}
\end{equation}
as $z\to \infty$. Define
\begin{equation} \label{an-eqn}
a_n = \Big( Y_{-1} \Big)_{12} \Big( Y_{-1} \Big)_{21}
\end{equation}
and
\begin{equation} \label{bn-eqn}
b_n = \frac{\Big( Y_{-2} \Big)_{12}}{ \Big( Y_{-1} \Big)_{12}} - \Big( Y_{-1} \Big)_{22},
\end{equation}
where the subscript $ij$ denotes  the $(i,j)$ entry of the corresponding matrix. If $a_n \neq 0$, then the monic orthogonal
polynomials $\pi_{n+1}$, $\pi_n$ and $\pi_{n-1}$ (of degrees $n+1$, $n$, and $n-1$) with respect to $w(z)$ on $\Sigma$, exist
and they satisfy the recurrence relation
\begin{equation}
\pi_{n+1}(z) = \Big( z- b_n \Big)\pi_{n}(z) - a_n \, \pi_{n-1}(z).
\end{equation}
\end{prop}
\begin{proof}
The proof is standard, see \cite{deift}. For the above precise form, see also \cite[Proposition 2.4]{dk}.
\end{proof}

As already said before, we are going to apply the Deift/Zhou steepest descent method to the above RH problem. In Sections
\ref{case1} we treat the case $\nu > 0$ and in Section \ref{case2} the case $\nu < 0$. Before we embark on the steepest descent
analysis let us say a few words about the method. The method consists of a number of explicit transformations, which in this
paper take the form $Y \mapsto U \mapsto T \mapsto S \mapsto R$, which lead to a ``simple'' RH problem for $R$. Since the works
of Deift et al. \cite{dkmvz,dkmvz2,deift} it is now clear what the main steps should be in the analysis of the RH problem
associated with orthogonal polynomials, namely
\begin{itemize}
\item normalization of the RH problem with the so-called $g$-function;
\item opening of the lens around the oscillatory region;
\item construction of global and local parametrices.
\end{itemize}
In any extension of the method (such as the one in this paper) these issues appear and it is unfortunate that all details have
to be checked each time, since we do not have, and maybe cannot expect, a master theorem giving general conditions under which
the method will work.

In the present paper we are dealing with orthogonality on a contour in the complex plane. To handle this feature we largely
follow the steepest descent analysis done for Laguerre polynomials with large negative parameters in \cite{km2} (for $\nu > 0$)
and in \cite{km1} (for $\nu < 0$). This will lead to an analysis on varying $n$-dependent curves depending on $A_n = 1 -
\frac{\nu}{n}$. For $n\to \infty$ the curves tend to the Szeg\H{o} curve. The precise analysis is rather delicate near the
critical point $z=1$, since we have to arrive at a local RH problem that can be modelled by the RH problem associated with
Painlev\'e IV. A simpler and maybe more natural approach would be to avoid the varying curves and work with the Szeg\H{o} curve
from the beginning. We tried to do this but we were not able to handle all difficulties this way.


\section{Case I: $\nu > 0$ and $\nu \not\in \mathbb{N}$} \label{case1}

\subsection{Introduction}

We assume $\nu > 0$ and $\nu \not\in \mathbb N$. Throughout we write
\begin{equation} \label{t-def-2}
    t:= \frac{N}{n}= 1 + \sqrt{2} \, L n^{-1/2}.
\end{equation}
Let
\begin{equation}
\label{a-def-2}
    A_n := - \frac{\alpha}{n}= 1 - \frac{\nu}{n}.
\end{equation}
Since we will concentrate on the asymptotics as $n \rightarrow \infty$, then without loss of generality throughout this paper we
assume $n$ is large enough. Thus, $A_n \in (0, 1)$. Following the Riemann-Hilbert steepest descent analysis in \cite{km2}, we
define
\begin{align}
    \label{beta1-def}
    \beta_{1,n} & := 2 - A_n - 2 \sqrt{1-A_n} = 1 + \frac{\nu}{n} - 2 \sqrt{\frac{\nu}{n}}, \\
    \label{beta2-def}
    \beta_{2,n} & := 2 - A_n + 2 \sqrt{1-A_n} = 1 + \frac{\nu}{n} + 2 \sqrt{\frac{\nu}{n}}
\end{align}
and
\begin{equation} \label{Rn-def}
    R_n(z) := \sqrt{(z-\beta_{1,n})(z-\beta_{2,n})}, \quad z \in \mathbb{C} \setminus [\beta_{1,n},\beta_{2,n}],
\end{equation}
where $R_n(z) \sim z$ as $z \to \infty$. Let
\begin{equation} \label{phi-1n}
    \phi_{n}(z) = \frac{1}{2} \int_{\beta_{1,n}}^z \frac{R_n(s)}{s} ds, \qquad
    z \in \mathbb{C} \setminus ( (-\infty,0] \cup [\beta_{1,n}, \infty) ),
\end{equation}
where the path of integration from $\beta_{1,n}$ to $z$ lies entirely in the region $\mathbb{C} \setminus ((-\infty,0] \cup
[\beta_{1,n}, \infty))$, except for the initial point $\beta_{1,n}$. The curves where $\re \phi_{n}(z)$ is constant are
trajectories of the quadratic differential (see Strebel \cite{strebel})
\begin{equation}
- \frac{R_n(s)^2}{s^2} ds^2  = - \frac{(s-\beta_{1,n})(s-\beta_{2,n})}{s^2} ds^2 ,
\end{equation}
which has two simple zeros at $\beta_{1,n}$ and $\beta_{2,n}$ and a double pole at 0. In \cite{km2} it is shown that there
exists a unique curve $\Gamma_{0,n}$ as follows:
\begin{defn} \label{gamma0}

The contour $\Gamma_{0,n}$ is a simple closed curve encircling 0 once, so that
\begin{equation} \label{phi-gamma0}
\re \phi_{n}(z) = 0 \qquad \hbox{for} \quad z \in \Gamma_{0,n};
\end{equation}
see Figure~\ref{contour-case01}.
\end{defn}

The curve $\Gamma_{0,n}$ depends on $n$ and it tends to the Szeg\H{o} curve $\mathcal{S}$ (\ref{szego-curve}) as $n \to \infty$.

\begin{lem} \label{gamma0-szego}
As $n \to \infty$, the curve $\Gamma_{0,n}$ tends to the Szeg\H{o} curve $\mathcal{S}$.
\end{lem}

\begin{proof}

From (\ref{beta1-def}) and (\ref{beta2-def}), we know that $\beta_{1,n}$ and $\beta_{2,n}$ both tend to 1 as $n \to \infty$.
Then, by (\ref{Rn-def}), $ R_n(z) \to  z - 1 $. As a consequence,
\begin{equation} \label{phi-1inf}
\phi_{n}(z) \to \frac{1}{2} \int_1^z \frac{s-1}{s} \, d s = \frac{1}{2} (z - 1 -\log{z}) \quad \textrm{as $ n \to \infty$},
\end{equation}
uniformly for $z$ bounded away from 0 and $\infty$. Therefore, from the definition of $\Gamma_{0,n}$ in Definition \ref{gamma0},
we know that $\Gamma_{0,n}$ tends to a simple closed curve $\Gamma_{0, \infty}$ encircling 0 once, and
\begin{equation}
\re(z - 1 -\log{z}) = 0 \qquad \textrm{for} \quad z \in \Gamma_{0,\infty}.
\end{equation}
From (\ref{szego-curve}), we see that $\Gamma_{0,\infty}$ is the Szeg\H{o} curve $\mathcal{S}$.
\end{proof}

In \cite{km2} a probability measure was defined for each $A < 1$ supported on a set $\Gamma_0 \cup [\beta_1, \beta_2]$, where
$\Gamma_0$ is a closed contour and $[\beta_1, \beta_2]$ an interval on the positive real line. In the present situation we have
that $A = A_n = 1 - \frac{\nu}{n}$ depends on $n$, but for each finite $n$ we still start from the corresponding probability
measure $\mu = \mu_{n}$ which now depends on $n$. The measure is given by
\begin{equation} \label{dmu-n}
    d \mu_{n}(y) = \frac{1}{2 \pi i} \frac{R_n(y)}{y} \chi_{\Gamma_{0,n}}(y) dy +
    \frac{1}{2 \pi i} \frac{R_{n,+}(y)}{y} \chi_{[\beta_{1,n},\beta_{2,n}]}(y) dy.
\end{equation}
Using (\ref{phi-1n}) and (\ref{phi-gamma0}), one can indeed check that (\ref{dmu-n}) defines a positive measure on $\Gamma_{0,n}
\cup [\beta_{1,n}, \beta_{2,n}]$.

\begin{figure}[h]
\centering
\includegraphics[width=200pt]{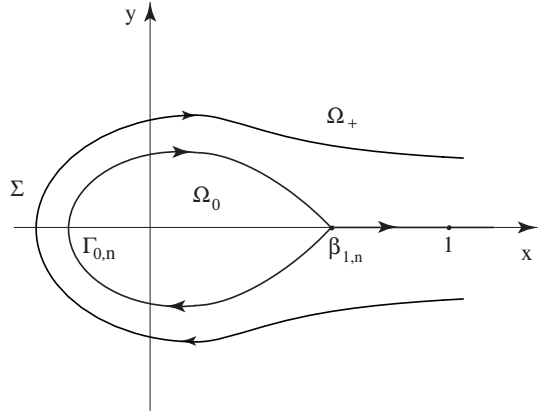}
\caption{Contours $\Sigma$ and $\Sigma^U = \Gamma_{0,n} \cup [\beta_{1,n}, \infty)$} \label{contour-case01}
\end{figure}

\subsection{First transformation $Y \mapsto U$}

Expecting the zero distribution as shown in Figure \ref{lagzero} for $0<A<1$, we start by modifying the contour in the RH
problem. Let
\begin{equation}
    \Sigma^U := \Gamma_{0,n} \cup [\beta_{1,n}, \infty).
\end{equation}
We assume (without loss of generality) that the contour $\Sigma$ in the RH problem for $Y$ was chosen so that $\Gamma_{0,n}$ is
contained in $\Omega_-$, see Figures \ref{Sigmacontour} and \ref{contour-case01}.

\begin{figure}[h]
\centering
\includegraphics[width=200pt]{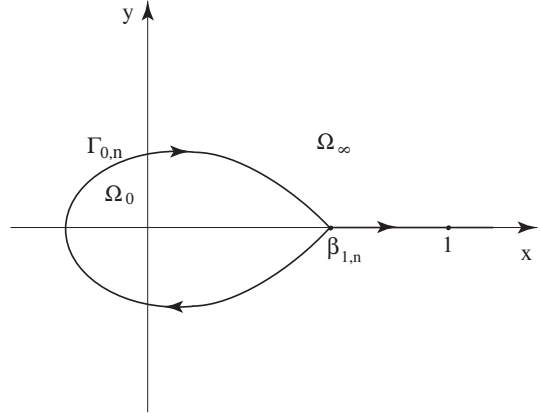}
\caption{The contour $\Sigma^U = \Gamma_{0,n} \cup [\beta_{1,n},\infty)$
    for the RH problem for $U$} \label{contour-case1}
\end{figure}

Introduce $U$ as
\renewcommand{\arraystretch}{1}
\begin{equation} \label{yy*}
    U(z) = \begin{cases} Y(z) & \textrm{for } z \in \Omega_+ \cup \Omega_0, \\
    Y(z) \left( \begin{matrix} 1 & w(z) \\ 0 & 1 \end{matrix} \right) & \textrm{for } z \in \Omega_- \setminus \Omega_0;
\end{cases}
\end{equation}
see Figure~\ref{contour-case01}. Then we obtain the following RH problem for $U$:

\begin{enumerate}
\item[(a)] $U(z)$ is analytic for $z \in \mathbb{C} \setminus
\Sigma^U$, see Figure~\ref{contour-case1}; \item[(b)] $U(z)$ possesses continuous boundary values on $\Sigma^U$ satisfying
\begin{align}
    U_+(z) & = U_-(z) \begin{pmatrix} 1 & z^{-n+ \nu} e^{-N z} (z-1)^{2b} \\ 0 & 1 \end{pmatrix}  \qquad
    \textrm{for } z \in \Gamma_{0,n}, \label{U-jump1} \\
    U_+(x) & = U_-(x) \begin{pmatrix} 1 & w_+(x) - w_-(x) \\ 0 & 1 \end{pmatrix}
     \qquad \textrm{for } x \in (\beta_{1,n} , 1) \cup (1, \infty); \label{U-jump2}
\end{align}
\item[(c)] $U(z)$ has the following behavior as $z \to \infty$
\begin{equation*}
    U(z)=\left( I + O\left(\frac{1}{z}\right) \right) \left( {\begin{array}{cc} z^n & 0 \\
    0 & z^{-n} \end{array}} \right);
\end{equation*}
\item[(d)] $U(z)$ has the following behavior as $z \to 1$:
\begin{equation} \label{u-0-behave}
    U(z) \left( \begin{matrix} 1 & c_{u, \pm} (z-1)^{2b} \\ 0 & 1 \end{matrix} \right) = O(1) \qquad \textrm{for } z \in
    \mathbb{C}^{\pm},
\end{equation}
where
\begin{equation} \label{cu-constant}
c_{u,+} = - e^{-N} \quad \hbox{and} \quad c_{u,-} = - e^{2 \nu \pi i -N}.
\end{equation}
\end{enumerate}
In (\ref{U-jump1}) and (\ref{u-0-behave}), the factor $(z-1)^{2b}$ is defined with a branch cut along $[1, \infty)$ (as before).

Because of the different choices of branches in the definition of $w(z)$ we find that
\begin{align} \label{wpm-def}
     w_+(x) - w_-(x) & = |x|^{-n + \nu} e^{-Nx} |x-1|^{2b} (1- e^{2\nu \pi i}), \qquad x \in (\beta_{1,n}, 1), \\
     w_+(x) - w_-(x) &= |x|^{-n + \nu} e^{-Nx} |x-1|^{2b} (1- e^{2(\nu + 2b) \pi i}), \qquad x \in (1, \infty). \label{wpm-def2}
\end{align}
Therefore, (\ref{U-jump2}) takes on different forms on the two intervals $(\beta_{1,n},1)$ and $(1,\infty)$.

Condition (d) is crucial in the above RH problem for $U$. With conditions (a)--(c) only, we do not have a unique solution for
our RH problem for $U$.  To obtain a unique solution, we have to add the extra condition (d) that controls the local behavior as
$z \to 1$. This local behavior (\ref{u-0-behave}) with the precise constants (\ref{cu-constant}) will also be important in the
construction of a local parametrix around $z=1$ that will ultimately be our goal.


\subsection{The $g$ and $\phi$ functions}

The measure $\mu_{n}$ in (\ref{dmu-n}) gives rise to the so-called $g$-function, which depends on  $n$,
\begin{equation} \label{gn}
    g_{n}(z) = \int\limits_{\Gamma_{0,n} \cup [\beta_{1,n},\beta_{2,n}]} \log(z-s) d \mu_{n}(s), \qquad
    z \in \mathbb{C} \setminus \Sigma^U,
\end{equation}
where the logarithm $\log(z-s)$ is defined with a branch cut along $\Sigma^U$, going to the right from $s$ to $\infty$.
Recalling the definition of $\phi_{n}(z)$ in (\ref{phi-1n}), we similarly define
\begin{equation}
    \tilde\phi_{n}(z) = \frac{1}{2} \int_{\beta_{2,n}}^z \frac{R_n(s)}{s} ds, \qquad
    z \in \mathbb{C} \setminus (-\infty, \beta_{2,n}].
\end{equation}
Note that
\begin{equation} \label{phi-phi}
\phi_{n}(z) = \tilde\phi_{n}(z) \pm \frac{\nu}{n} \pi i \qquad \textrm{for } z \in \mathbb{C}^{\pm};
\end{equation}
see \cite[Lemma 3.4.2]{dev}. We also introduce
\begin{equation} \label{gt}
    g_{t,n}(z) := g_n(z) + (t-1) g_n^0(z)
\end{equation}
and
\begin{equation}  \label{phi-tz}
\phi_{t,n}(z) := \phi_n(z) + (t-1) \phi_n^0(z), \qquad \tilde\phi_{t,n}(z) := \tilde\phi_n(z) + (t-1) \phi_n^0(z),
\end{equation}
where
\begin{equation}
    g_n^0(z) =  \begin{cases} \frac{1}{2} (z - R_n(z)), & z \in \Omega_\infty,
    \\ \frac{1}{2} (z + R_n(z)), & z \in \Omega_0, \end{cases}
\end{equation}
and
\begin{equation} \label{phi0}
    \phi_n^0(z) = \frac{R_n(z)}{2}, \qquad z \in \mathbb{C} \setminus [\beta_{1,n}, \beta_{2,n}].
\end{equation}
Note that $g_n^0(z) = O(1/z)$ as $z \to \infty$. It is easily seen that the functions $ g_n^0(z)$ and $\phi_n^0(z)$ satisfy the
following properties.

\begin{prop} \label{g-prop-2}
\begin{enumerate}
\item[\rm (a)] We have
\begin{equation} \label{g&phi-2}
    g_n^0(z) = \frac{1}{2} \, z +
    {\begin{cases} - \phi_n^0(z), \quad z \in \Omega_\infty \cap
    \mathbb{C}^{\pm}, \\ + \phi_n^0(z), \quad z \in \Omega_0 \cap \mathbb{C}^{\pm}.
    \end{cases}}
\end{equation}

\item[\rm (b)] For $z \in \Sigma^U$, we have
\begin{equation}
    g^0_{n,+}(z)-g^0_{n,-}(z) = \begin{cases} - 2\phi^0_{n}(z), & z \in
    \Gamma_{0,n} \cap \mathbb{C}_\pm,\\
    -2\phi^0_{n,+}(z) = 2 \phi^0_{n,-}(z),& z\in(\beta_{1,n},\beta_{2,n}].
\end{cases}
\end{equation}
\item[\rm (c)] We also have
\begin{equation}
    g_{n,+}^0(z) + g_{n,-}^0(z) = \begin{cases}  z, & z \in \Gamma_{0,n} \cup [\beta_{1,n},\beta_{2,n}] \text{,}\\
    z  -2\phi^0_{n}(z)\text{,} & z \in [\beta_{2,n},\infty)\text{.}
\end{cases}
\end{equation}
\end{enumerate}
\end{prop}

Then from the above proposition and Proposition 3.2 in \cite{km2}, we get the following properties for $g_{t,n}$, $\phi_{t,n}$,
and $\tilde{\phi}_{t,n}$.
\begin{prop} \label{g-prop}
\begin{enumerate}
\item[\rm (a)] There exists a constant $l_{t,n}$ such that
\begin{equation} \label{g&phi}
    g_{t,n}(z) = \frac{1}{2} \left( A_n\log z + t z + l_{t,n} \right) +
    {\begin{cases} \mp \frac{1}{2} A_n \pi i - \phi_{t,n}(z), \quad z \in \Omega_\infty \cap
    \mathbb{C}^{\pm}, \\ \pm \frac{1}{2} A_n \pi i + \phi_{t,n}(z), \quad z \in \Omega_0 \cap \mathbb{C}^{\pm}.
    \end{cases}}
\end{equation}
Here $\log z$ is defined with a cut along $[0, \infty)$. The constant $l_{t,n}$ is explicitly given by
\[ l_{t,n} = 2g_{t,n}(x_n) - (A_n \log x_n + x_n), \]
where $x_n$ is the intersection of $\Gamma_{0,n}$ with the negative real axis.

\item[\rm (b)] For $z \in \Sigma^U$, we have
\begin{equation}
    g_{t,n,+}(z)-g_{t,n,-}(z) = \begin{cases} - \phi_{t,n}(z) - \tilde\phi_{t,n}(z) \mp \pi i, & z \in
    \Gamma_{0,n} \cap \mathbb{C}_\pm,\\
    -2\phi_{t,n,+}(z)-2 A_n \pi i = 2 \phi_{t,n,-}(z)-2 A_n \pi i\text{,}& z\in(\beta_{1,n},\beta_{2,n}]\text{,}\\
    -2 \pi i \text{,}& z \in [\beta_{2,n},\infty).
\end{cases}
\end{equation}
\item[\rm (c)] We also have
\begin{equation}
    g_{t,n,+}(z) + g_{t,n,-}(z) = \begin{cases} A_n \log{z}+ t z + l_{t,n} , & z \in \Gamma_{0,n}\text{,}\\
    A_n \log{z} + t z + A_n \pi i + l_{t,n} \text{,}& z \in [\beta_{1,n},\beta_{2,n}]\text{,}\\
    A_n \log{z} + t z + A_n \pi i + l_{t,n} -2\tilde\phi_{t,n}(z)\text{,} & z \in [\beta_{2,n},\infty)\text{.}
\end{cases}
\end{equation}
\end{enumerate}
\end{prop}


\subsection{Second transformation $U \mapsto T$}

To normalize the RH problem at infinity, introduce the second transformation $U \mapsto T$ as
\begin{multline} \label{t-y}
    T(z) =  \Big( (-1)^n (e^{- i \pi \nu} - e^{i \pi \nu})  \Big)^{-\frac{1}{2} \sigma_3}
        e^{-\frac{1}{2} n l_{t,n} \sigma_3} U(z) \\
        \times e^{ - n g_{t,n}(z) \sigma_3} e^{\frac{1}{2} n l_{t,n} \sigma_3} \Big( (-1)^n (e^{- i \pi \nu}
     - e^{i \pi \nu}) \Big)^{\frac{1}{2} \sigma_3}.
\end{multline}
Note that it is important here that $\nu$ is not an integer, since this assumption guarantees that $ e^{-i \pi \nu} - e^{i\pi
\nu} \neq 0$ and therefore $T$ is well-defined.

With the RH problem for $U$ and Propositions \ref{g-prop-2} and \ref{g-prop}, direct calculation gives us a RH problem for $T$
as follows. As before $(z-1)^{2b}$ is defined with a cut along $[1, \infty)$.
\begin{enumerate}
\item[(a)] $T(z)$ is analytic for $z \in \mathbb{C} \setminus \Sigma^U$;

\item[(b)]
\begin{align}
    T_+(z) = \ & T_-(z)
    \begin{pmatrix} (-1)^n e^{n (\phi_{t,n}(z) + \tilde\phi_{t,n}(z))} & (-1)^n (e^{-\nu \pi i} - e^{\nu \pi i})^{-1}  (z-1)^{2b} \\
     0 & (-1)^n e^{-n (\phi_{t,n}(z) + \tilde\phi_{t,n}(z))}
    \end{pmatrix},    z \in \Gamma_{0,n}; \label{t-jump0} \\
    T_+(z) = \ & T_-(z)  \begin{pmatrix} e^{2n \tilde\phi_{t,n,+}(z)} &  (z-1)^{2b} \\ 0 & e^{2n \tilde\phi_{t,n,-}(z)}
    \end{pmatrix}, \hspace{75pt} z \in (\beta_{1,n}, 1);  \label{t-jump1} \\
    T_+(z) = \ & T_-(z) \begin{pmatrix} e^{2n \tilde\phi_{t,n,+}(z)} &  |z-1|^{2b}
    \frac{\sin(\nu + 2b) \pi}{\sin{\nu \pi}} e^{2b \pi i} \\
    0 & e^{2n \tilde\phi_{t,n,-}(z)}
    \end{pmatrix},  \quad \ z \in (1, \beta_{2,n}); \label{t-jump2} \\
    T_+(z) = \ & T_-(z)  \begin{pmatrix} 1 &  e^{-2n \tilde\phi_{t,n}(z)} |z-1|^{2b}
    \frac{\sin(\nu + 2b) \pi}{\sin{\nu \pi}} e^{2b \pi i} \\
    0 & 1 \end{pmatrix},  \quad z \in (\beta_{2,n}, \infty); \label{t-jump3}
\end{align}

\item[(c)] $T(z)$ has the following behavior as $z \to \infty$
\begin{equation*}
    T(z) = I + O \left( \frac{1}{z} \right);
\end{equation*}

\item[(d)] $T(z)$ has the following behavior as $z \to 1$:
\begin{equation}
T(z) \left( \begin{matrix} 1 & c_{t, \pm}(z-1)^{2b} \\ 0 & 1 \end{matrix} \right) = O (1) \qquad \textrm{for } z \in
\mathbb{C}^{\pm}
\end{equation}
with
\begin{equation} \label{ct-constant}
c_{t, \pm} = (-1)^{n} (e^{- \nu \pi i} - e^{\nu \pi i})^{-1}   \, e^{2n g_{t,n,\pm}(1) - n l_{t,n} } c_{u, \pm},
\end{equation}
where $c_{u,\pm}$ is given in (\ref{cu-constant}).
\end{enumerate}

To derive the jump matrices for $T$ in (\ref{t-jump1})--(\ref{t-jump3}), one needs to make use of (\ref{U-jump2}),
(\ref{wpm-def}), (\ref{wpm-def2}) and Proposition \ref{g-prop}. The conjugation with the constant matrix $\Big( (-1)^n (e^{- i
\pi \nu} - e^{i \pi \nu}) \Big)^{-\frac{1}{2} \sigma_3}$ in (\ref{t-y}) may look a bit awkward since it makes the jump matrix in
(\ref{t-jump0}) more complicated. We have introduced it in order to simplify the jump matrices (\ref{t-jump1})--(\ref{t-jump3}).
We also want to point out that the entry $\frac{\sin(\nu + 2b) \pi}{\sin{\nu \pi}}$ in (\ref{t-jump2}) and (\ref{t-jump3}) will
come out in the Stokes multipliers $s_j$'s in the RH problem for PIV later on.

By using the relations between the  $g$-functions and $\phi$-functions in (\ref{g&phi}) and (\ref{g&phi-2}), the constant $c_{t,
\pm}$ in (\ref{ct-constant}) can be rewritten as
\begin{equation} \label{ct-constant2}
    c_{t, \pm} = - (e^{- \nu \pi i} - e^{\nu \pi i})^{-1} \,
    e^{- 2n\phi_{t,n,\pm}(1) \pm \nu \pi i}.
\end{equation}

\subsection{Third transformation $T \mapsto S$}

From (\ref{phi-gamma0}), (\ref{phi-phi}) and (\ref{phi-tz}), we know that, when $|t-1|$ is small, the real part of
$\phi_{t,n}(z) + \tilde\phi_{t,n}(z)$ is close to 0 for $z \in \Gamma_{0,n}$. This means the jump matrix for $T$ on
$\Gamma_{0,n}$ in (\ref{t-jump0}) has oscillatory diagonal entries. To remove this oscillating jump, we introduce the third
transformation and open the lens as follows. In a fixed neighborhood of $z = 1$ let the curve $\Gamma_1$ be defined as $\im z =
\arg z$, then we continue $\Gamma_1$ such that it becomes a closed contour with the Szeg\H{o} curve $\mathcal{S}$ in its
interior. Note that $\Gamma_1$ is independent of $n$ and $t$. Then $\Gamma_1$ and $\Gamma_{0,n}$ divide the complex plane into
three domains $\Omega_0$, $\Omega_1$ and $\Omega_{\infty}^S$; see Figure~\ref{contourt1}.
\begin{figure}[h]
\centering
\includegraphics[width=220pt]{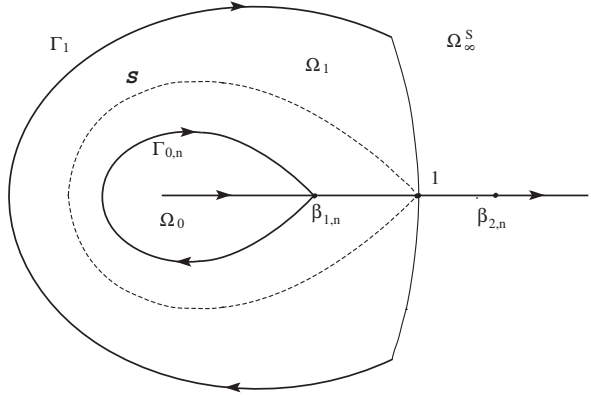}
\caption{Szeg\H{o} curve and the contour for the RH problem for $S$} \label{contourt1}
\end{figure}
Define
\begin{eqnarray}
S(z) & = & T(z) \begin{pmatrix} 0 & (-1)^n \frac{(z-1)^{2b}}{e^{-\nu \pi i} - e^{\nu \pi i}}  \\
    (-1)^{n+1} \frac{e^{-\nu \pi i} - e^{\nu \pi i}}{(z-1)^{2b}} & (-1)^n e^{-n(\phi_{t,n}(z) + \tilde\phi_{t,n}(z))}
    \end{pmatrix}, \ z \in \Omega_0; \label{t-s1}\\
S(z) & = & T(z) \begin{pmatrix} 1 & 0 \\ - \frac{e^{-\nu \pi i} - e^{\nu \pi i}}{(z-1)^{2b}} \, e^{n(\phi_{t,n}(z) +
\tilde\phi_{t,n}(z))} & 1 \end{pmatrix}, \hspace{65pt}  z \in \Omega_1; \label{t-s2}\\
S(z) & = & T(z), \hspace{240pt} z \in \Omega_\infty^S. \label{t-s3}
\end{eqnarray}
Direct calculation shows that $S_+(z) = S_-(z)$ for $z \in \Gamma_{0,n}$. Therefore, $S(z)$ has an analytic continuation across
$\Gamma_{0,n}$ and we obtain a RH problem for $S$ on the contour $\Sigma^S = \Gamma_1 \cup [0, \infty)$ shown in
Figure~\ref{contourt2}. We use $\Omega_0^S$ to denote the bounded domain enclosed by $\Gamma_1$.

\begin{figure}[h]
\centering
\includegraphics[width=200pt]{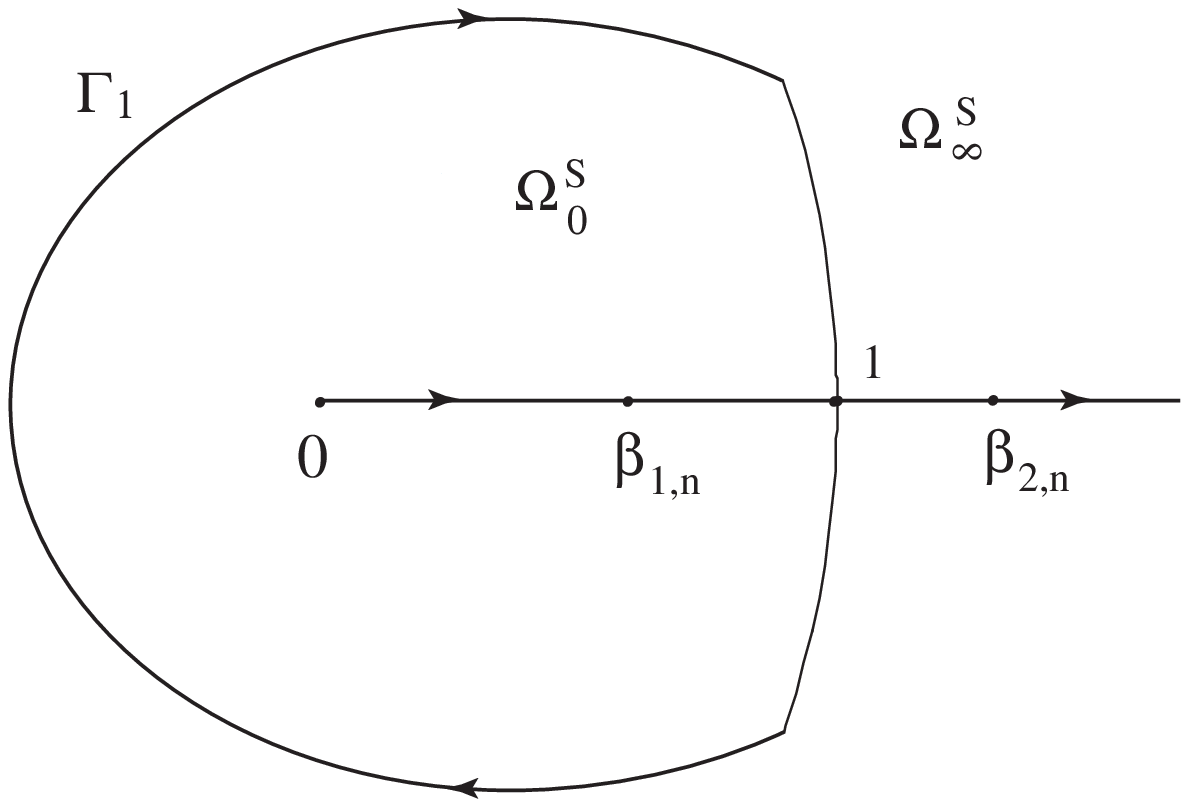}
\caption{The contour $\Sigma^S$ for the RH problem for $S$} \label{contourt2}
\end{figure}

\begin{enumerate}

\item[(a)] $S$ is analytic for $z \in \mathbb{C} \setminus
\Sigma^S$, see Figure~\ref{contourt2};

\item[(b)]
\begin{align} \label{s-jump1}
S_+(z) = \ & S_-(z) \left( \begin{matrix} 1 & 0 \\ \frac{e^{-\nu \pi i} - e^{\nu \pi i}}{(z-1)^{2b}}  e^{n(\phi_{t,n}(z) +
\tilde\phi_{t,n}(z))} & 1 \end{matrix} \right), &  z \in \Gamma_1; \\
S_+(z) = \ & S_-(z) \left( \begin{matrix} 1 &  (z-1)^{2b} e^{-2n \phi_{t,n}(z)} \\ 0 & 1 \end{matrix} \right), &  z \in
(0, \beta_{1,n}); \\
S_+(z) = \ & S_-(z) \left( \begin{matrix} e^{2n \phi_{t,n,+}(z)} &  (z-1)^{2b} \\ 0 & e^{2n \phi_{t,n,-}(z)} \end{matrix}
\right), &
z \in (\beta_{1,n},1); \\
S_+(z) = \ & S_-(z) \left( \begin{matrix} e^{2n \tilde\phi_{t,n,+}(z)} &  |z-1|^{2b}
\frac{\sin(\nu + 2b) \pi }{\sin{\nu \pi}} e^{2b \pi i} \\
0 & e^{2n \tilde\phi_{t,n,-}(z)}
\end{matrix} \right), &  z \in (1, \beta_{2,n}); \\ \label{s-jump5}
S_+(z) = \ & S_-(z) \left( \begin{matrix} 1 &  |z-1|^{2b}  \frac{\sin(\nu + 2b) \pi }{\sin{\nu \pi}} e^{2b \pi i}
e^{-2n \tilde\phi_{t,n}(z)} \\
0 & 1 \end{matrix} \right), & \hspace{-10pt}  z \in (\beta_{2,n}, \infty);
\end{align}
\item[(c)] $S(z)$ has the following behavior as $z \to \infty$
\begin{equation*}
S(z) = I + O\left(\frac{1}{z}\right);
\end{equation*}

\item[(d)] $S(z)$ has the following behavior as $z \to 1$
\begin{eqnarray}
    S(z) \begin{pmatrix} 1 & c_{t,\pm}(z-1)^{2b} \\ -c_{t,\pm}^{-1}(z-1)^{-2b} & 0 \end{pmatrix}
    & = & O (1), \ z \in \Omega_0^S \cap \mathbb{C}^{\pm}, \label{s-1-behave} \\
    S(z) \begin{pmatrix} 1 & c_{t,\pm}(z-1)^{2b} \\ 0 & 1 \end{pmatrix}& = &
    O (1), \ z \in \Omega_\infty^S \cap \mathbb{C}^{\pm}, \label{s-1-behave2}
\end{eqnarray}
where $c_{t,\pm}$ is given in (\ref{ct-constant2}).
\end{enumerate}

The jump matrices in (\ref{s-jump1})--(\ref{s-jump5}) look quite complicated. However, most of them tend to the identity matrix
as $n \to \infty$ at an exponential rate. To clarify this fact, we need the following proposition.

\begin{figure}[h]
\centering
\includegraphics[width=200pt]{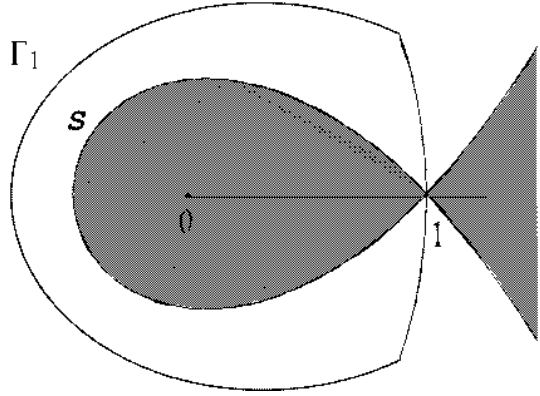}
\caption{$\mathcal{S}$ is the Szeg\H{o} curve. The dark region is the region  where $\re (z - 1 -\log{z}) > 0$, the white region
is the region where $\re (z - 1 -\log{z}) < 0$.} \label{contourt2phi}
\end{figure}

\begin{prop} \label{b-delta}

Let $\delta$ be a fixed small constant  and $B_\delta : = \{ z \in \mathbb{C} \ | \ |z-1|< \delta \}$. Then there exist $\eta>0$
and $\varepsilon > 0$ such that for all $t \in \mathbb{R}$ with $|t-1|<\eta$ and all $n$ large enough, we have
\[ \re \phi_{t,n}(z) < - \varepsilon \qquad
    \text{for } z \in \Gamma_1 \setminus B_\delta \]
 and
 \[ \re \phi_{t,n}(z) > \varepsilon ( |z| + 1 ) \qquad
    \text{for } z \in (0,\infty) \setminus B_\delta. \]
The same inequalities hold for $\re \tilde\phi_{t,n}(z)$.

\end{prop}
\begin{proof}

First recall from (\ref{phi-1inf}), that  $\phi_{n}(z) \to \frac{1}{2} (z - 1 -\log{z})$ uniformly for $z$ bounded away from 0
and $\infty$; see Figure \ref{contourt2phi} for the property of $\re(z - 1 -\log{z})$. For $z$ near 0 and $\infty$ we have
\begin{equation*}
\phi_{n}(z) = \begin{cases} - \log{z} + O(1) & \hbox{as } z \to 0, \\ z + O(1) & \hbox{as } z \to \infty, \end{cases}
\end{equation*}
both uniformly for $n$ large enough. Since both the contour $\Gamma_1$ and the constant $\delta$ are independent of $n$, there
exists an $\varepsilon > 0 $ such that $\re \phi_{n}(z) < - \varepsilon$ on $\Gamma_1 \setminus B_\delta$ and $\re \phi_{n}(z) >
\varepsilon ( |z| + 1 )$ on $(0,\infty) \setminus B_\delta$ uniformly for $n$ large enough. From the fact that
\begin{equation*}
\phi_n^0(z) = \frac{R_n(z)}{2} \to  \frac{z-1}{2} \qquad \hbox{as } n \to \infty, \quad \hbox{unifomly for all } z,
\end{equation*}
and the definition of $\phi_{t,n}(z)$ in (\ref{phi-tz}), it follows by continuity that the same inequalities hold for
$\phi_{t,n}(z)$ if $t$ is sufficiently close to 1.

Since $\re \phi_{n}(z) = \re \tilde\phi_{n}(z)$ for all $z \in \mathbb{C}$, the inequalities also hold if we replace $\phi_n(z)$
by $\tilde{\phi}_n(z)$.
\end{proof}


\subsection{Construction of the local parametrix}

\subsubsection{RH problem for $P$} From Proposition
\ref{b-delta}, we know that the jump matrices for $S$ are exponentially close to the identity matrix if $n$ is large, except for
the ones in the neighborhood of $z=1$. Now, let us focus on the RH problem of $S$ restricted to a neighborhood $B_\delta $ of
$z=1$.  We seek a $2 \times 2$ matrix valued function $P$ that satisfies the following RH problem.
\begin{enumerate}
\item[(a)] $P(z)$ is analytic for $z \in B_\delta \setminus
\Sigma^S$, and continuous on $\overline{B}_\delta \setminus \Sigma^S$;

\item[(b)] $P(z)$ satisfies the same jump conditions on $\Sigma^S
\cap B_\delta$ as $S$ does; see (\ref{s-jump1})--(\ref{s-jump5});

\item[(c)] on $\partial B_\delta$, as $n \to \infty$
\begin{equation} \label{p-0condition}
    P(z) = \left( I + O \left( \frac{1}{\sqrt{n}} \right) \right) n^{\frac{b}{2} \sigma_3} \qquad
    \textrm{uniformly for } z \in \partial B_\delta \setminus
    \Sigma^S;
\end{equation}

\item[(d)] $P(z)$ satisfies the same local behavior near $z=1$ as
$S$ does; see (\ref{s-1-behave}) and (\ref{s-1-behave2}) with the same constants $c_{t,\pm}$ as given in (\ref{ct-constant2}).
\end{enumerate}

The factor $n^{ \frac{b}{2} \sigma_3}$ in the matching condition (\ref{p-0condition}) is special and unusual in the local
parametrix construction. It comes out of our construction with the RH problem for $\Psi$ from Section \ref{rh-piv} and we are
not able to remove or simplify it. However, we can deal with the factor in the final transformation in Section \ref{r-section}
below.


\subsubsection{Reduction to constant jumps} \label{rhp-q}
We look for $P$ in the following form
\begin{equation} \label{pp1}
    P(z) = \begin{cases} \sigma_1 Q(z) \sigma_1
    e^{(n \phi_{t,n}(z) -\frac{1}{2}(\nu + b)\pi i) \sigma_3} (z-1)^{-b \sigma_3}
\quad \textrm{for } z \in B_\delta \cap \Omega_0^S, \\
\sigma_1 Q(z) \sigma_1 e^{(n \tilde\phi_{t,n}(z) + \frac{1}{2}(\nu + b)\pi i) \sigma_3} (1-z)^{-b \sigma_3} \quad \textrm{for }
z \in B_\delta \cap \Omega_\infty^S,
\end{cases}
\end{equation}
where $\sigma_1 = \left(\begin{matrix} 0 & 1 \\ 1 & 0
\end{matrix} \right)$, and $(z-1)^b$ and $(1-z)^b$ are defined with branch
cuts along $[1, \infty)$ and $(-\infty,1]$, respectively. Thus we have $0 < \arg(z-1) < 2\pi$, $0 < \arg(1-z) < 2\pi$ and
\begin{equation} \label{z-1-cuts}
\arg(z-1) = \arg(1-z) \mp \pi i \quad \textrm{ for } z \in \mathbb{C}^{\pm}.
\end{equation}
The factors $e^{-n \phi_{t,n}(z) \sigma_3 } (z-1)^{b \sigma_3}$ and $e^{-n \tilde\phi_{t,n}(z) \sigma_3} (1-z)^{b \sigma_3}$ in
(\ref{pp1}) are introduced to cancel the corresponding factors in the jump matrices in (\ref{s-jump1})--(\ref{s-jump5}) and make
them independent of $z$. The result is that $Q$ will satisfy a RH problem with constant jumps.

Reorient the contour $\Sigma^S \cap B_{\delta}$ so that all parts are oriented away from the point $z=1$; see
Figure~\ref{sigmaq}. Then it can be verified that, in order for $P(z)$ to satisfy the desired RH problem, $Q(z)$ should satisfy
a RH problem as follows:
\begin{figure}[h]
\centering
\includegraphics[width=130pt]{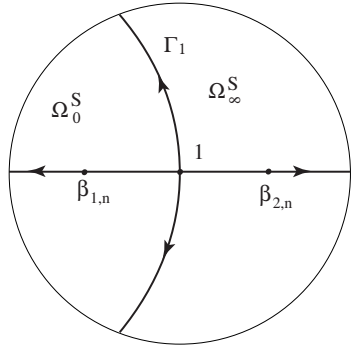}
\caption{The contour $\Sigma^S \cap B_\delta$ for the RH problem for $Q$} \label{sigmaq}
\end{figure}
\begin{enumerate}

\item[(a)] $Q(z)$ is analytic for $z \in B_\delta \setminus
\Sigma^S$, and continuous on $\overline{B}_\delta \setminus \Sigma^S$;

\item[(b)]
\begin{align}
    Q_+(z) = \ & Q_-(z) \begin{pmatrix} 1 & 0 \\ \frac{\sin(\nu + 2b) \pi}{\sin{\nu \pi}} e^{(\nu + b) \pi i} & 1
    \end{pmatrix}, \hspace{85pt}  z \in (1, 1 + \delta); \label{q-jump1} \\
    Q_+(z) = \ & Q_-(z) \begin{pmatrix} 1 & (e^{\nu \pi i} - e^{- \nu \pi i}) e^{b \pi i}  \\ 0 & 1
    \end{pmatrix}, \hspace{57pt} z \in \Gamma_1 \cap B_\delta \cap
    \mathbb{C}^+;  \\
    Q_+(z) = \ & Q_-(z) \begin{pmatrix} 1 & 0 \\ -e^{-(\nu + b)\pi i}  & 1 \end{pmatrix},
    \hspace{112pt} z \in (1- \delta,1); \label{q-jump3} \\
    Q_+(z) = \ & Q_-(z) \begin{pmatrix} 1 & (e^{-\nu \pi i} - e^{\nu \pi i}) e^{(2\nu + b) \pi i} \\ 0 & 1
    \end{pmatrix} e^{2(\nu + b) \pi i \sigma_3}, \ z \in \Gamma_1 \cap B_\delta \cap \mathbb{C}^-;
    \label{q-jump4}
\end{align}

\item[(c)] for $z \in \partial B_\delta$ as $n \to \infty$,
\begin{align}
    Q(z) & = \left( I + O\left( \frac{1}{\sqrt{n}} \right) \right) (\sqrt{n}(z-1))^{-(\nu + b) \sigma_3}
    \exp\biggl[\frac{n}{2} (z-1- \log{z}) \nonumber \\
    & \hspace{15pt} + \frac{n^{\frac{1}{2}}  \, L }{\sqrt{2}} (z-1)
    + \frac{\nu}{2} (\log{z} + \log{\nu} - 1) + \frac{1}{2}(\nu - b)\pi i  \biggr] \sigma_3 \label{q-match}
\end{align}
where $\log{z}$ and $(z-1)^{-(\nu + b)}$ are defined with cuts along $(-\infty, 0]$ and along $\Gamma_1 \cap B_\delta \cap
\mathbb{C}^-$, respectively;

\item[(d)] $Q(z)$ has the following behavior  as $z \to 1$:
\begin{align}
    Q(z) \left( \begin{matrix} 0 & (e^{-\nu \pi i} - e^{\nu \pi i}) e^{b \pi i}
    \\  \frac{e^{-b \pi i}}  {e^{\nu \pi i} - e^{-\nu \pi i}}
    & 1 \end{matrix} \right) (z-1)^{b \sigma_3} & =
    O (1), & z \in \Omega_0^S \cap \mathbb{C}^+; \label{q-0-behave1} \\
    Q(z) \left( \begin{matrix} 0 & (e^{-\nu \pi i} - e^{\nu \pi i}) e^{b \pi i}
    \\  \frac{e^{-b \pi i}}{e^{\nu \pi i} - e^{-\nu \pi i}}
    & e^{-2 \nu \pi i} \end{matrix} \right) (z-1)^{b \sigma_3} & =
    O (1), & z \in \Omega_0^S \cap \mathbb{C}^-; \\
    Q(z) \left( \begin{matrix} 1 & 0 \\ \frac{e^{-b \pi i}}{e^{\nu \pi i} - e^{-\nu \pi i}}
    & 1 \end{matrix} \right) (z-1)^{b \sigma_3} & =
    O (1), & z \in \Omega_\infty^S \cap \mathbb{C}^+; \label{q-0-behave3} \\
    Q(z) \left( \begin{matrix} 1 & 0 \\ \frac{e^{(2 \nu + 3b) \pi i}}{e^{\nu \pi i} - e^{-\nu \pi i}}
    & 1 \end{matrix} \right) (z-1)^{b \sigma_3} & =
    O (1), & z \in \Omega_\infty^S \cap \mathbb{C}^-
    \label{q-0-behave4};
\end{align}
where now $(z-1)^{b}$ is defined with a cut along $\Gamma_1 \cap B_\delta \cap \mathbb{C}^-$.
\end{enumerate}

In the above RH problem for $Q$, it is notable that, although $\beta_{2,n} \in (1, 1+\delta)$, the jump matrix in
(\ref{q-jump1}) is the same for $z < \beta_{2,n}$ and $z > \beta_{2,n}$. This is in contrast to the RH problem for $P$, where we
have different expressions for $z < \beta_{2,n}$ and $z
> \beta_{2,n}$. A similar thing happens for the jump matrix in (\ref{q-jump3}).

To obtain the matching condition in (\ref{q-match}), we need to make use of the following asymptotic formulas for $\phi_n^0(z)$
and $\phi_{n}(z)$ as $n \to \infty$. For $z \in \partial B_\delta$, we have from (\ref{phi0})
\begin{equation} \label{phi0a}
\phi_n^0(z) = \frac{1}{2}(z-1) - \frac{\nu (z+1)}{2n(z-1)} + O \left( \frac{1}{n^2} \right)
\end{equation}
and from (\ref{phi-1n}), calculated with the assistance of Maple,
\begin{align}
\phi_{n}(z) & = \frac{1}{2}(z - 1 -\log{z}) - \frac{\nu}{2n} \log{n} + \frac{\nu}{2n}
\Big( -2 \log(z-1) + 2 \pi i + \log{z} + \log{\nu} - 1 \Big) \nonumber \\
& \quad + \frac{\nu^2}{n^2 (z-1)^2} + O \left( \frac{1}{n^3} \right), \label{phi1a}
\end{align}
where $\log{z}$ and $\log(z-1)$ are defined with cuts along $(-\infty,0]$ and $[1, \infty)$, respectively; see also \cite{dev}
and \cite{kmm}. For $z \in
\partial B_\delta \cap \Omega_0^S$, we have from (\ref{phi-tz}), (\ref{p-0condition}) and (\ref{pp1})
\begin{equation*}
Q(z)  = \left( I + O(\frac{1}{\sqrt{n}}) \right) (\sqrt{n}(z-1))^{-b \sigma_3} e^{(n (\phi_{n}(z) + (t-1) \phi_n^0(z))
-\frac{1}{2}(\nu + b)\pi i) \sigma_3}.
\end{equation*}
Substituting (\ref{phi0a}) and (\ref{phi1a}) into the above formula, we get
\begin{equation*}
\begin{split}
Q(z) & = \left( I + O(\frac{1}{\sqrt{n}}) \right) (\sqrt{n}(z-1))^{-(\nu + b) \sigma_3} \exp\biggl[\frac{n}{2} (z- 1 - \log{z})  \\
& \hspace{15pt} + \frac{\nu}{2} (2 \pi i + \log{z} + \log{\nu} - 1) + O(\frac{1}{n}) + \frac{n(t-1)}{2}\left((z-1)+
O(\frac{1}{n})\right) - \frac{1}{2}(\nu + b)\pi i \biggr] \sigma_3,
\end{split}
\end{equation*}
where $\log{z}$ and $(z-1)^{-(\nu + b)}$ are defined with cuts along $(-\infty, 0]$ and $\Gamma_1 \cap B_\delta \cap
\mathbb{C}^-$, respectively. Recalling the definition of $t$ in (\ref{t-def-2}), the above formula immediately gives us
(\ref{q-match}) for $z \in
\partial B_\delta \cap \Omega_0^S$. Similarly, we can prove that (\ref{q-match})
also holds for $z \in \partial B_\delta \cap \Omega_\infty^S$.

To obtain the limiting behaviors (\ref{q-0-behave1})--(\ref{q-0-behave4}) near $ z= 1$, some careful calculations are needed. We
take (\ref{q-0-behave4}) as an example and (\ref{q-0-behave1})--(\ref{q-0-behave3}) can be obtained in a similar way. From
(\ref{pp1}), (\ref{z-1-cuts}) and condition (d) for the RH problem for $P$, we have, as $z \to 1, z \in \Omega_\infty^S \cap
\mathbb{C}^-$
\begin{equation}\label{q-match-step1}
Q(z)  e^{-(n \tilde\phi_{t,n,-}(1) + \frac{1}{2}(\nu + b)\pi i) \sigma_3} ((z-1) e^{- \pi i})^{b \sigma_3}
\left( \begin{matrix} 1 & 0 \\
c_{t,-}(z-1)^{2b} & 1 \end{matrix} \right) = O (1),
\end{equation}
where $(z-1)^{b}$ is defined with a cut along $[1,\infty)$ and $c_{t,-}$ is given in (\ref{ct-constant2}). Changing the branch
cut of $(z-1)^{b} $ to $\Gamma_1 \cap B_\delta \cap \mathbb{C}^-$, we get
\begin{equation} \label{z-1-branch}
(z-1) \mapsto (z-1) e^{2 \pi i} \qquad \textrm{for } z \in \Omega_\infty^S \cap \mathbb{C}^-.
\end{equation}
From (\ref{q-match-step1}) and (\ref{z-1-branch}), it follows that, as $z \to 1, z \in \Omega_\infty^S \cap \mathbb{C}^-$
\begin{equation*}
Q(z)  e^{-(n \tilde\phi_{t,n,-}(1) + \frac{1}{2}(\nu + b)\pi i) \sigma_3} ((z-1) e^{ \pi i})^{b \sigma_3}
\left( \begin{matrix} 1 & 0 \\
c_{t,-}(z-1)^{2b} e^{4 b \pi i} & 1 \end{matrix} \right) = O (1).
\end{equation*}
Right multiplying $e^{(n \tilde\phi_{t,n,-}(1) + \frac{1}{2}(\nu + b)\pi i) \sigma_3} e^{ -b \pi i \sigma_3}$ on both sides of
the above equation yields, as $z \to 1, z \in \Omega_\infty^S \cap \mathbb{C}^-$
\begin{equation*}
Q(z) (z-1)^{b \sigma_3} \left( \begin{matrix} 1 & 0 \\
c_{t,-} (z-1)^{2b}  e^{ 2n \tilde\phi_{t,n,-}(1) + (\nu + 3b) \pi i} & 1
\end{matrix} \right)  = O (1).
\end{equation*}
Recalling (\ref{phi-phi}), (\ref{ct-constant2}) and moving $(z-1)^{b \sigma_3}$ to the right of the lower triangular matrix, we
obtain (\ref{q-0-behave4}).


\subsubsection{Comparison with the RH problem for $\Psi$} \label{case1-construct1}

Now we compare the RH problem for $Q$ with the RH problem for $\Psi$ in Section \ref{rh-piv}. From
(\ref{q-jump1})--(\ref{q-jump4}) we see that we need the Stokes multipliers
\renewcommand{\arraystretch}{1.5}
\begin{equation} \label{case1-sm}
\begin{array}{ll}
s_1 = \displaystyle\frac{\sin(\nu + 2b) \pi}{\sin{\nu \pi}} e^{(\nu + b) \pi i}, & s_2 =(e^{\nu \pi i} - e^{-\nu \pi i})
e^{b \pi i}, \\
s_3 = -e^{-(\nu + b) \pi i}, & s_4 = (e^{-\nu \pi i} - e^{\nu \pi i}) e^{(2 \nu + b ) \pi i}.
\end{array}
\end{equation}
\renewcommand{\arraystretch}{1}
Because of the factors $(z-1)^{b \sigma_3}$ in (\ref{q-0-behave1})--(\ref{q-0-behave4}) and $(z-1)^{- (\nu + b) \sigma_3} $ in
(\ref{q-match}), we are led to choose
\begin{equation} \label{theta-def}
\Theta = - b, \qquad \Theta_{\infty} = \nu + b.
\end{equation}
Then the above Stokes multipliers are obtained from those in (\ref{case0-sm}), by applying the transformation in
(\ref{sm-trans}) with
\[ d = (e^{\nu \pi i} - e^{- \nu \pi i })^{-1}. \]
Again it is important that $\nu$ is not an integer, since this ensures $d \neq \infty$. Furthermore, from (\ref{c-explicit-2}),
we get
\begin{equation*}
\mathcal{C}^{-1} = \left(
\begin{matrix} 1 & 0 \\ d e^{(2 \Theta_\infty - \Theta) \pi i} & 1
\end{matrix} \right) = \left( \begin{matrix} 1 & 0 \\ \frac{e^{(2 \nu + 3b) \pi i}}{e^{\nu \pi i} - e^{-\nu \pi i}}
& 1 \end{matrix} \right),
\end{equation*}
which appears exactly in (\ref{q-0-behave4}). Thus the jump matrices and the local behavior near the point $1$ in the RH problem
for $Q$ correspond exactly to the jump matrices and the behavior near the origin in the RH problem for $\Psi$.

From (\ref{q-0-behave1})--(\ref{q-0-behave4}), we do not want a logarithmic singularity at point 1 for $Q$ for all values of
$b$. This is special when we want to construct our parametrix for $Q$ by using the $\Psi$\,-function in Section \ref{rh-piv}.
Note that $\l = 0$ is a Fuchsian singular point of the equation $\D\frac{\partial \Psi}{\partial \l} = A \Psi$ in
(\ref{PIVsystem}) and the leading term of $A$ as $\l \to 0$ is
\begin{equation*}
\frac{1}{\l} \left[ (\Theta - K) \sigma_3 - \frac{u \, y}{2} \sigma_+ + \frac{2 K}{ u \, y} (K - 2 \Theta) \sigma_- \right].
\end{equation*}
The eigenvalues of the above coefficient matrix are $\Theta$ and $-\Theta$. When $\Theta = m/2, m \in \mathbb{Z}$, usually this
is a resonant case and there is a logarithmic singularity at 0 for $\Psi$; see \cite[Section 1.3]{fikn}. In this paper, we are
going to make use of $\Psi$\,-function which is free from logarithmic singularities. Then with $\Theta = -b$, $\Theta_\infty =
\nu + b$ and the Stokes multipliers given in (\ref{case1-sm}), in Section \ref{special} we will see that there exist special
function solutions to PIV when $b = m/2, m \in \mathbb{Z}$. Also see \cite{fn} for similar cases for Painlev\'e II solutions.

\subsubsection{Construction of $Q$} \label{case1-construct2}

Using the RH problem for $\Psi$, we now construct $Q$ as follows. We use the mapping function
\begin{equation} \label{f-def}
    f(z)= [z - 1 - \log z]^{1/2}
\end{equation}
which is a conformal map from a neighborhood of $1$ onto a neighborhood of $0$. We have
\begin{equation} \label{f-expansion}
    f(z) = \frac{1}{\sqrt{2}}(z-1) - \frac{\sqrt{2}}{6}(z-1)^2 +
    O((z-1)^3) \qquad \textrm{as } z \to 1.
\end{equation}
Since for $z \in \Gamma_1 \cap B_{\delta}$ we have that $\im z = \arg z$, it can be shown that $f$ maps $\Gamma_1 \cap B_\delta$
to the imaginary axis.

Now we define
\begin{equation} \label{p1con}
Q(z) =  E(z) \Psi \left(n^{\frac{1}{2}} f(z), L \frac{z-1}{\sqrt{2} f(z)}\right)
\end{equation}
where
\begin{equation} \label{e-def}
E(z) = \left( \frac{f(z)}{z-1} \right)^{(\nu  + b) \sigma_3} (\nu e^{-1} \; z)^{\frac{\nu}{2} \sigma_3} e^{\frac{1}{2} (\nu - b
) \pi i \sigma_3},
\end{equation}
$L$ is given in (\ref{alpha-N-def}) and $\Psi(\cdot, \cdot)$ is the solution of the RH problem (\ref{psi-rhb})--(\ref{psi-0})
with $\Theta, \Theta_\infty$ given in (\ref{theta-def}) and Stokes multipliers $s_j$ given in (\ref{case1-sm}). Because of
(\ref{f-expansion}) we have that $E$ is analytic in a (small enough) neighborhood of $z=1$. We also see that
\[  L \frac{z-1}{\sqrt{2} f(z)} \to L \qquad \textrm{as } z \to 1,
\]
Since $L$ is (by assumption) not a pole of the PIV solution $u(s)$, there is a small enough $\delta > 0$ (depending on $L$) so
that
\[ z \mapsto s = L \frac{z-1}{\sqrt{2} f(z)} \]
maps $\overline{B}_{\delta}$ to a compact subset of the $s$-plane that does not contain any poles of $u(s)$. Then $\Psi(\lambda,
L \frac{z-1}{\sqrt{2} f(z)})$ exists for all $z \in B_{\delta}$, and $Q(z)$ is well-defined and analytic for $z \in B_{\delta}
\setminus \Sigma^S$.

It is now easy to check that $Q$ satisfies the jump conditions (\ref{q-jump1})--(\ref{q-jump4}). We also find that the behavior
as $z \to 1$ in (\ref{q-0-behave4}) is satisfied. This follows from the corresponding behavior as $\lambda \to 0$ in the RH
problem for $\Psi$. The behaviors as $z \to 1$ in (\ref{q-0-behave1})--(\ref{q-0-behave3}) follow from this and the jump
conditions in the RH problem for $\Psi$.

So what remains is to check the matching condition (\ref{q-match}) in the RH problem for $Q$. For this we need to note that the
asymptotic condition (\ref{psi-asy}) holds uniformly for $s$ in compact sets away from the poles of $u(s)$. Thus by the
definition (\ref{p1con}) we have for $z \in \partial B_{\delta}$ as $n \to \infty$
\begin{equation*}
Q(z) = E(z) \left( I + O\left(\frac{1}{\sqrt{n}}\right) \right) \exp\biggl[\frac{n}{2} f^2(z) \, \sigma_3 +
\frac{n^{\frac{1}{2}} \sqrt{2} \, L (z-1)}{2} \, \sigma_3 \biggr] \left( \frac{1}{n^{1/2} f(z)} \right)^{(\nu+ b) \sigma_3},
\end{equation*}
where $f(z)^{-(\nu + b)}$ is defined with a cut along $\Gamma_1 \cap B_\delta \cap \mathbb{C}^-$. Since $E(z)$ is independent of
$n$, then
\begin{equation*}
Q(z) = \left( I + O\left(\frac{1}{\sqrt{n}}\right) \right) \exp\biggl[\frac{n}{2} f^2(z) \, \sigma_3 + \frac{n^{\frac{1}{2}}
\sqrt{2} \, L (z-1)}{2} \, \sigma_3 \biggr] \left( \frac{1}{n^{1/2} f(z)} \right)^{(\nu+ b) \sigma_3} E(z).
\end{equation*}
From the definition of $E(z)$ in (\ref{e-def}), we get
\begin{align}
Q(z) & = \left( I + O\left(\frac{1}{\sqrt{n}}\right) \right) (\sqrt{n}(z-1))^{-(\nu + b) \sigma_3}
\exp\biggl[\frac{n}{2} f^2(z) + \frac{n^{\frac{1}{2}} \sqrt{2} \, L (z-1)}{2} \nonumber \\
& \quad + \frac{\nu}{2} (\log{z} + \log{\nu} - 1) + \frac{1}{2}(\nu - b)\pi i  \biggr] \sigma_3, \label{q-largen}
\end{align}
where $\log{z}$ and $(z-1)^{-(\nu + b)}$ are defined with cuts along $(-\infty, 0]$ and $\Gamma_1 \cap B_\delta \cap
\mathbb{C}^-$, respectively. Recalling the definition of $f(z)$ in (\ref{f-def}), we obtain (\ref{q-match}) from
(\ref{q-largen}).

This completes the construction of $Q$.


\subsection{Final transformation} \label{r-section}

Now let us consider the difference between the exact solution $S$ and the parametrix $P$ constructed. Define
\begin{equation} \label{r-s}
R(z) = \begin{cases} n^{\frac{b}{2} \sigma_3} S(z) n^{-\frac{b}{2} \sigma_3},
& z \in \mathbb{C} \setminus (\overline{B}_\delta \cup \Sigma^S), \\
n^{\frac{b}{2} \sigma_3} S(z) P^{-1}(z), & z \in B_\delta \setminus \Sigma^S.
\end{cases}
\end{equation}
From the RH problems for $S$ and $P$, since $P$ is analytic for $z \in B_\delta \setminus \Sigma^S$ and satisfies the same jump
conditions on $\Sigma^S \cap B_\delta$ as $S$ does, $R(z)$ is analytic for $z \in B_\delta$ except for a possible pole at $z =
1$. Moreover, because $S$ and $P$ have the same local behavior near $z=1$, $R(z)$ has to be analytic at $z=1$. Thus, we get a RH
problem for $R(z)$ on a contour $\Sigma^R$ as follows:

\begin{figure}[h]
\centering
\includegraphics[width=200pt]{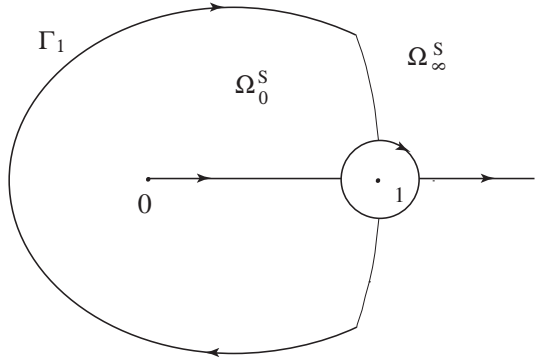}
\caption{Contour $\Sigma^R$ for the RH problem for $R$} \label{contourr}
\end{figure}

\begin{enumerate}

\item[(a)] $R(z)$ is analytic for $z \in \mathbb{C} \setminus \Sigma^R$;
see Figure~\ref{contourr};

\item[(b)] $R_+(z) = R_-(z) J_R(z)$ for $z \in \Sigma^R$;

\item[(c)] $R(z) = I + O(1/z)$ as $z \to \infty$;

\item[(d)] $R(z) = O(1)$ as $z \to 1$.

\end{enumerate}
The jump matrices in the above RH problem are given by
\begin{eqnarray}
J_R(z) & = & \left( \begin{matrix} 1 & 0 \\ \frac{e^{-\nu \pi i} - e^{\nu \pi i}}{(z-1)^{2b}} n^{-b}  e^{n(\phi_{t,n}(z) +
\tilde\phi_{t,n}(z))}
& 1 \end{matrix} \right), \hspace{60pt}  z \in \Gamma_1 \setminus B_\delta, \nonumber \\
J_R(z) & = & \left( \begin{matrix} 1 & n^b  (z-1)^{2b} e^{-2n \phi_{t,n}(z)} \\ 0 & 1 \end{matrix} \right), \hspace{85pt} z
\in (0,1- \delta), \nonumber \\
J_R(z) & = & \left( \begin{matrix} 1 & n^b  |z-1|^{2b}  \frac{\sin(\nu + 2b) \pi }{\sin{\nu \pi}}
e^{2b \pi i} e^{-2n \tilde\phi_{t,n}(z)} \\
0 & 1 \end{matrix} \right), \qquad z \in (1+ \delta, \infty), \nonumber \\
J_R(z) & = & P(z) n^{-\frac{b}{2} \sigma_3}, \hspace{200pt} z \in \partial B_\delta.
\end{eqnarray}
From Proposition \ref{b-delta}, the jump matrices are exponentially close to the identity matrix if $n$ is large except for the
one on $\partial B_\delta$. For $\partial B_\delta$, we have the following lemma:
\begin{lem} \label{lem-jr}

The jump matrix $J_R(z)$ on $\partial B_\delta$ is given by
\begin{equation} \label{p-asy}
J_R(z) = P(z) n^{-\frac{b}{2} \sigma_3} = I + \frac{1}{\sqrt{n}} P^{(-1)}(z) + \frac{1}{n} P^{(-2)}(z) + O(n^{-\frac{3}{2}}),
\end{equation}
uniformly for $z \in \partial B_\delta$, where $P^{(-1)}(z)$ and $P^{(-2)}(z)$ are given by, for $z \in B_\delta$
\begin{equation} \label{p-1}
P^{(-1)}(z) = \frac{1}{f(z)} E^{-1}(z) \sigma_1 \Psi_{-1}\left( L \frac{z-1}{\sqrt{2} f(z)} \right) \sigma_1 E(z) - \frac{\nu \,
L (z+1)}{\sqrt{2} \, (z-1)} \sigma_3
\end{equation}
and
\begin{equation} \label{p-2}
\begin{split}
P^{(-2)}(z) & = \frac{1}{f^2(z)} E^{-1}(z) \sigma_1 \Psi_{-2}\left( L \frac{z-1}{\sqrt{2} f(z)} \right) \sigma_1 E(z) +
\frac{\nu^2 \, {L}^2 (z+1)^2}{4
(z-1)^2} I \\
& \quad + \frac{\nu^2}{(z-1)^2} \sigma_3 - \frac{\nu \, L (z+1)}{\sqrt{2} \, (z-1) f(z)} E^{-1}(z) \sigma_1 \Psi_{-1}\left( L
\frac{z-1}{\sqrt{2} f(z)} \right) \sigma_1 E(z) \sigma_3.
\end{split}
\end{equation}
Here $E(z)$ is given by (\ref{e-def}), $\Psi_{-1}$ and $\Psi_{-2}$ are given in (\ref{psi-asy}).
\end{lem}
\begin{proof}

Since $Q(z) = E(z) \Psi\left(n^{\frac{1}{2}} f(z), L \frac{z-1}{\sqrt{2} f(z)}\right)$, we have from (\ref{phi-tz}) and
(\ref{pp1})
\begin{equation} \label{p&psi}
P(z)=  \sigma_1 E(z) \Psi\left(n^{\frac{1}{2}} f(z), L \frac{z-1}{\sqrt{2} f(z)} \right) \sigma_1
    e^{(n (\phi_{n}(z) + (t-1) \phi_n^0(z)) -\frac{1}{2}(\nu + b)\pi i) \sigma_3} (z-1)^{-b \sigma_3}
\end{equation}
for $z \in B_\delta \cap \Omega_0^S$. Using the formulas for $\phi_n^0(z)$ and $\phi_{n}(z)$ in (\ref{phi0a})--(\ref{phi1a}),
the definition of $E(z)$ in (\ref{e-def}) and the asymptotic formula for $\Psi$ in (\ref{psi-asy}), we get from (\ref{p&psi})
\begin{align}
P(z) n^{-\frac{b}{2} \sigma_3} = & E^{-1}(z) \Big[ I + \frac{\sigma_1 \Psi_{-1}\left( L \frac{z-1}{\sqrt{2} f(z)} \right)
\sigma_1 }{\sqrt{n} f(z)} + \frac{\sigma_1 \Psi_{-2}\left( L \frac{z-1}{\sqrt{2} f(z)} \right) \sigma_1}{n f^2(z)}
 + O(n^{-\frac{3}{2}}) \Big] \nonumber \\
& \times  E(z) \; \exp \left[ - \frac{\nu \, L (z+1)}{\sqrt{2n} \, (z-1)} \sigma_3 + \frac{\nu^2}{n \, (z-1)^2} \sigma_3
+O(n^{-\frac{3}{2}}) \right],
\end{align}
which gives us (\ref{p-1}) and (\ref{p-2}) for $z \in B_\delta \cap \Omega_0^S$. Similarly, one can prove (\ref{p-1}) and
(\ref{p-2}) also hold for $z \in B_\delta \cap \Omega_\infty^S$.
\end{proof}

For large $n$, the jump matrix $J_R$ is close to the identity matrix, both in $L^\infty$ and in $L^2$-sense on $\Sigma^R$. Then
following similar analysis as in \cite{dkmvz2, kmvv}, for large enough $n$, we know the RH problem for $R$ is solvable.
Moreover, from (\ref{p-asy}), $R(z)$ can be expanded as follows
\begin{equation} \label{r-nasy}
R(z) = I +   \frac{1}{\sqrt{n}} R^{(-1)}(z) + \frac{1}{n} R^{(-2)}(z) + O(n^{-\frac{3}{2}}),
\end{equation}
as $n \to \infty$, uniformly for $z \in \mathbb{C} \setminus \Sigma^R$. To prove our Theorem \ref{thm1}, we need to derive
explicit asymptotic formulas for $ R^{(-1)}(z)$ and $R^{(-2)}(z)$ as $z \to \infty$.


\subsection{Asymptotic formulas for $ R^{(-1)}(z)$ and $R^{(-2)}(z)$}

To derive asymptotic formulas for $ R^{(-1)}(z)$ and $R^{(-2)}(z)$ as $z \to \infty$, more information about $\Psi_{-1}(s)$ and
$\Psi_{-2}(s)$ in (\ref{psi-asy}) is required. Using Proposition 1.1 in \cite[p.51]{fikn}, one can derive the explicit formulas
for them as follows:
\begin{equation} \label{psi-1}
\Psi_{-1}(s) = \left( \begin{matrix} -H & -\frac{y}{2} \\ \frac{1}{y}(K - \Theta - \Theta_\infty) & H \end{matrix} \right),
\end{equation}
\begin{equation} \label{psi-2}
\Psi_{-2}(s) = \left( \begin{matrix} (\Psi_{-2}(s))_{11}  & (\Psi_{-2}(s))_{12} \\
(\Psi_{-2}(s))_{21} & (\Psi_{-2}(s))_{22}
\end{matrix} \right),
\end{equation}
where
\begin{eqnarray}
(\Psi_{-2}(s))_{11} & = & \frac{1}{2} \Big[H^2 + s \, H - \frac{K}{2} - \frac{1}{2} (\Theta - \Theta_\infty -1) (\Theta +
\Theta_\infty) \Big], \label{psi-2-1} \\
(\Psi_{-2}(s))_{12} & = & \frac{y}{2} \left(\frac{u}{2} + s - H \right), \\
(\Psi_{-2}(s))_{21} & = & \frac{K}{u \, y}(K - 2 \Theta) - \frac{s+H}{y} (K - \Theta - \Theta_\infty), \\
(\Psi_{-2}(s))_{22} & = & \frac{1}{2} \Big[H^2 - s \, H - \frac{K}{2} + \frac{1}{2} (\Theta - \Theta_\infty + 1) (\Theta +
\Theta_\infty) \Big], \label{psi-2-4}
\end{eqnarray}
$y$ and $K$ are given in (\ref{y-def}) and (\ref{k-def}), respectively, and
\begin{equation} \label{h-def}
H(s) = H := \frac{K}{u} (K - 2 \Theta) - \left( \frac{u}{2} + s \right)(K - \Theta -\Theta_\infty).
\end{equation}
Then, we obtain the following expansion for $R^{(-1)}(z)$ as $z \to \infty$.

\begin{lem} \label{r-1-lemma}
We have that
\begin{equation} \label{r-1-nasy}
R^{(-1)}(z) = \frac{M}{z-1} = \frac{M}{z} + \frac{M}{z^2} + O(z^{-3}) \qquad \hbox{as } z \to \infty,
\end{equation}
where
\begin{equation} \label{p-1-residue}
M:= \Res\limits_{z=1} P^{(-1)}(z) = \sqrt{2} \left( \begin{matrix}  H(L) - \nu \, L & \frac{1}{y(L)}(K(L) - \nu) \rho^{-2} \\
-\frac{1}{2} \, y(L) \rho^2 & -H(L) + \nu \, L
\end{matrix} \right)
\end{equation}
and
\begin{equation} \label{rho}
\rho = 2^{-(\nu + b)/2} \, (\nu e^{-1})^{\nu/2} \ e^{\frac{1}{2} (\nu - b) \pi i}.
\end{equation}

\end{lem}

\begin{proof}
From the RH problem for $R$ and the expansions in (\ref{p-asy}) and (\ref{r-nasy}), we know that $R^{(-1)}(z)$ satisfies a RH
problem as follows:
\begin{enumerate}

\item[(a)] $R^{(-1)}(z)$ is analytic for $z \in  \mathbb{C} \setminus \partial B_\delta$;

\item[(b)] $R^{(-1)}_+(z) = R^{(-1)}_-(z) + P^{(-1)}(z) $ for $z \in \partial B_\delta$;

\item[(c)] $R^{(-1)}(z) = O(1/z)$ as $z \to \infty$.

\end{enumerate}
Here $\partial B_\delta$ is clockwise oriented. From (\ref{p-1}) and the property of $f(z)$ in (\ref{f-expansion}), we know that
$P^{(-1)}(z)$ is analytic in $B_\delta$ except for a simple pole at 1. Thus, $R^{(-1)}(z)$ is explicitly given by:
\begin{equation} \label{r-1-explicit}
R^{(-1)}(z) = \begin{cases} \displaystyle \frac{1}{z-1} \Res\limits_{z=1} P^{(-1)}(z),
\hspace{80pt} z \in \mathbb{C} \setminus \overline{B}_\delta, \\
\displaystyle \frac{1}{z-1} \Res\limits_{z=1} P^{(-1)}(z) -  P^{(-1)}(z), \qquad z \in B_\delta.
\end{cases}
\end{equation}
This formula immediately gives us (\ref{r-1-nasy}). From (\ref{p-1}), (\ref{psi-1}) and the definition of $E(z)$ in
(\ref{e-def}), we get (\ref{p-1-residue}) with $\rho$ given by (\ref{rho}).
\end{proof}

Similarly, we get the expansion for $R^{(-2)}(z)$ as $z \to \infty$.
\begin{lem}

We have that
\begin{equation} \label{r-2-nasy}
R^{(-2)}(z) = \frac{B_{-2}}{(z-1)^2} + \frac{B_{-1}}{z-1} = \frac{B_{-1}}{z} + \frac{B_{-1} + B_{-2}}{z^2} + O(z^{-3}) \qquad
\hbox{as } z \to \infty,
\end{equation}
where $B_{-1}$ is a constant matrix and $B_{-2}$ is given by
\begin{equation} \label{b-2}
2 \left( \begin{matrix} (\Psi_{-2}(L))_{22} - \nu L H(L) +  \D\frac{ \nu ^2({L}^2 + 1)}{2} & \Big( (\Psi_{-2}(L))_{21} +
\frac{\nu L}{y(L)} (K(L) - \nu ) \Big) \rho^{-2} \\ \Big( (\Psi_{-2}(L))_{12} + \frac{\nu L}{2} y(L) \Big) \rho^2 &
(\Psi_{-2}(L))_{11} - \nu L H(L) + \D\frac{\nu^2 ({L}^2 -1)}{2}
\end{matrix} \right)
\end{equation}
with $\rho$ given by (\ref{rho}) and $(\Psi_{-2}(s))_{ij}$ given in (\ref{psi-2-1})--(\ref{psi-2-4}).

\end{lem}

\begin{proof}

As in the proof of Lemma \ref{r-1-lemma}, we have a RH problem for $R^{(-2)}(z)$ from (\ref{p-asy}) and (\ref{r-nasy}):

\begin{enumerate}

\item[(a)] $R^{(-2)}(z)$ is analytic for $ z \in \mathbb{C} \setminus
\partial B_\delta$;

\item[(b)] $R^{(-2)}_+(z) = R^{(-2)}_-(z) + P^{(-2)}(z) +
P^{(-1)}(z) (R^{(-1)})_-(z) $ for $z \in \partial B_\delta$;

\item[(c)] $R^{(-2)}(z) = O(1/z)$ as $z \to \infty$.

\end{enumerate}
From (\ref{f-expansion}) and (\ref{p-2}), we see that $P^{(-2)}(z)$ is analytic in $B_\delta$ except for a double pole at 1. If
we rewrite $P^{(-2)}(z) + P^{(-1)}(z) (R^{(-1)})_-(z)$ as
\begin{equation} \label{p-2-rewrite}
P^{(-2)}(z) + P^{(-1)}(z) (R^{(-1)})_-(z) = \frac{B_{-2}}{(z-1)^2} + \frac{B_{-1}}{z-1} + W(z),
\end{equation}
where $B_{-1}$ and $B_{-2}$ are two constant matrices, $W(z)$ is an analytic function in $B_\delta$. Then, $R^{(-2)}(z)$ can be
given as follows
\begin{equation}
R^{(-2)}(z) = \begin{cases} \displaystyle \frac{B_{-2}}{(z-1)^2} + \frac{B_{-1}}{z-1},
\qquad z \in \mathbb{C} \setminus \overline{B}_\delta, \\
-  W(z), \hspace{73pt} z \in B_\delta.
\end{cases}
\end{equation}
Thus, as $z \to \infty$, we have
\begin{equation}
R^{(-2)}(z) = \left( \frac{1}{z^2} + O (z^{-3}) \right) B_{-2} + \left( \frac{1}{z} + \frac{1}{z^2} + O (z^{-3}) \right) B_{-1},
\end{equation}
which gives us (\ref{r-2-nasy}). From  (\ref{p-2}), (\ref{psi-2}), (\ref{p-2-rewrite}) and the definition of $E(z)$ in
(\ref{e-def}), we have (\ref{b-2}).
\end{proof}


\subsection{Proof of Theorem \ref{thm1} in Case I} \label{thm1-proof}

Then we are ready for the proof of Theorem \ref{thm1}.
\begin{proof}

Reversing the transformations $Y \mapsto U \mapsto T \mapsto S \mapsto R $, we can recover $Y$ from $R$. Since the RH problem
for $R$ is solvable for large enough $n$, it follows that the RH problem $Y$ has a solution for large enough $n$. This implies
that the corresponding monic orthogonal polynomial $\pi_n(z)$ uniquely exists for large enough $n$.

Next, we are going to calculate the asymptotic formulas for the recurrence coefficients $a_n$ and $b_n$. From the definition of
$g_{n}$ in (\ref{gn}), we have
\begin{equation*}
\begin{split}
g_{n}(z) & = \int \log(z-s) d \mu_{n}(s) \\
& = \log z - \frac{1}{z} \int s \ d \mu_{n}(s) - \frac{1}{2 z^2} \int s^2 \ d \mu_{n}(s) + O (z^{-3})
\end{split}
\end{equation*}
and so by (\ref{gt})
\[ g_{t,n}(z) =  g_n(z) + (t-1) g_n^0(z) =
    \log z - d_1 z^{-1} - d_2 z^{-2} + O(z^{-3}) \qquad \textrm{as } z \to \infty,
    \]
 for certain constants $d_1$ and $d_2$ (that depend on $n$ and $t$).
 Then, from (\ref{t-y}), (\ref{t-s3}) and (\ref{r-s}), we get
\begin{align}
R(z) & = \Big( (-1)^n (e^{- i \pi \nu} - e^{i \pi \nu}) n^{-b} \Big)^{-\frac{1}{2} \sigma_3} e^{-\frac{1}{2} n l_{t,n} \sigma_3}
\left( I + \frac{Y_{-1}}{z} + \frac{Y_{-2}}{z^2} + O(z^{-3}) \right) \nonumber \\
& \quad \times \left( I + \frac{d_1 n}{z} \sigma_3 + \frac{d_1^2 n^2}{2 z^2} I + \frac{d_2 n}{z^2} \sigma_3 + O(z^{-3}) \right)
e^{\frac{1}{2} n l_{t,n} \sigma_3} \Big( (-1)^n (e^{- i \pi \nu} -
e^{i \pi \nu}) n^{-b} \Big)^{\frac{1}{2} \sigma_3} \nonumber \\
& = I + \frac{R_{-1}}{z} + \frac{R_{-2}}{z^2} + O(z^{-3}) \label{r-z-exp}
\end{align}
with
\begin{eqnarray}
R_{-1} & = & c^{-\sigma_3} \Big(  Y_{-1} + d_1 n \sigma_3 \Big) c^{\sigma_3}, \\
R_{-2} & = & c^{-\sigma_3} \Big( Y_{-2} + d_1 n \, Y_{-1} \sigma_3 + \frac{d_1^2 n^2}{2} I + d_2 n \sigma_3 \Big) c^{\sigma_3},
\end{eqnarray}
where $c = \Big( (-1)^n (e^{- i \pi \nu} - e^{i \pi \nu}) n^{-b} \Big)^{\frac{1}{2}} e^{\frac{1}{2} n l_{t,n}}$. Recalling
(\ref{an-eqn}) and (\ref{bn-eqn}), it immediately follows that
\begin{equation} \label{an-r}
a_n = \Big( Y_{-1} \Big)_{12} \Big( Y_{-1} \Big)_{21}  = \Big( R_{-1} \Big)_{12} \Big( R_{-1} \Big)_{21}
\end{equation}
and
\begin{eqnarray}
b_n & = & \frac{\Big( Y_{-2} \Big)_{12}}{ \Big( Y_{-1} \Big)_{12} } - \Big( Y_{-1} \Big)_{22} = \frac{\Big( R_{-2} \Big)_{12}
c^2 + d_1 n \, \Big( Y_{-1} \Big)_{12}}{ \Big( R_{-1} \Big)_{12} c^2} -
\left[ \Big( R_{-1} \Big)_{22} + d_1 n \right] \nonumber \\
& = & \frac{\Big( R_{-2} \Big)_{12} c^2 + d_1 n \, \Big( R_{-1} \Big)_{12} c^2}{ \Big( R_{-1} \Big)_{12} c^2} -
\left[ \Big( R_{-1} \Big)_{22} + d_1 n \right] \nonumber \\
& = & \frac{\Big( R_{-2} \Big)_{12}}{ \Big( R_{-1} \Big)_{12}} - \Big( R_{-1} \Big)_{22}. \label{bn-r}
\end{eqnarray}
From (\ref{r-nasy}), (\ref{r-1-nasy}) and (\ref{r-2-nasy}), we know  that
\begin{align}
R(z) & =  I + \frac{1}{\sqrt{n}} \left( \frac{1}{z} + \frac{1}{z^2} + O(z^{-3}) \right) \, M
\nonumber \\
& \quad + \frac{1}{n} \left( \frac{B_{-1}}{z} + \frac{B_{-1} + B_{-2}}{z^2} + O(z^{-3}) \right) + O(n^{-\frac{3}{2}}).
\label{r-z-exp2}
\end{align}
Combining (\ref{r-z-exp}) and (\ref{r-z-exp2}), it follows that
\begin{equation} \label{r-1}
R_{-1} = \frac{1}{\sqrt{n}} M + \frac{1}{n} \, B_{-1} + O(n^{-\frac{3}{2}})
\end{equation}
and
\begin{equation} \label{r-2}
R_{-2} = \frac{1}{\sqrt{n}} M + \frac{1}{n} \, \Big( B_{-1} + B_{-2} \Big) + O(n^{-\frac{3}{2}}).
\end{equation}
Thus, from (\ref{an-r}), we get
\begin{equation}
a_n = \frac{1}{n} \Big( M \Big)_{12} \; \Big( M \Big)_{21} + O(n^{-\frac{3}{2}}).
\end{equation}
From (\ref{p-1-residue}) and the assumption that $L$ is not a pole of $u(s)$ and $K(L) \neq \nu$ in Theorem \ref{thm1},  we have
$(M)_{12} = \frac{\sqrt{2}}{y(L)}(K(L) - \nu) \rho^{-2} \neq 0$. Then, from (\ref{bn-r}), we have
\begin{align}
b_n & = \frac{ \Big( M \Big)_{12} + \frac{1}{\sqrt{n}} \, \Big( ( B_{-1} )_{12} + ( B_{-2} )_{12} \Big) + O(n^{-1})}{ \Big( M
\Big)_{12} + \frac{1}{\sqrt{n}} \, ( B_{-1} )_{12} + O(n^{-1})} - \frac{1}{\sqrt{n}} \Big( M \Big)_{22} + O
( n^{-1} ) \nonumber \\
& = \left[ 1 + \frac{1}{\sqrt{n}} \,  \frac{( B_{-1} )_{12}}{\Big( M \Big)_{12}} + \frac{1}{\sqrt{n}} \,  \frac{( B_{-2}
)_{12}}{\Big( M \Big)_{12}} + O(n^{-1}) \right] \left[ 1 - \frac{1}{\sqrt{n}} \,  \frac{( B_{-1} )_{12}}{\Big( M \Big)_{12}} +
O(n^{-1}) \right] \nonumber \\
& \quad - \frac{1}{\sqrt{n}} \Big( M \Big)_{22} + O ( n^{-1} ) \nonumber \\
& = 1 + \frac{1}{\sqrt{n}} \,  \frac{( B_{-2} )_{12}}{( M )_{12}} - \frac{1}{\sqrt{n}} \Big( M \Big)_{22} + O ( n^{-1} ).
\end{align}
With (\ref{p-1-residue}) and (\ref{b-2}), the above formulas give us (\ref{an-final}) and (\ref{bn-final}).

This completes the proof of Theorem \ref{thm1} for Case I.
\end{proof}


\section{Case II: $\nu < 0$} \label{case2}

\subsection{Introduction}

Since $\nu < 0$, with the definition of $A_n$ in (\ref{a-def-2}) we know $A_n > 1$. Following the RH analysis in \cite{km1},
define
\begin{equation}
\beta_n := 2 - A_n + 2 i \sqrt{A_n -1}
\end{equation}
and
\begin{equation}
R_n(z):= \sqrt{(z- \beta_n)( z - \overline{\beta}_n)}, \qquad z \in \mathbb{C} \setminus \Gamma_{0,n},
\end{equation}
where $R_n(z)$ behaves like $z$ as $z \to \infty$.  Recalling the value of $A_n$ in (\ref{a-def-2}), $\beta_n$ can be rewritten
as
\begin{equation} \label{beta3-def}
\beta_n = 1 + \frac{\nu}{n} + 2i \sqrt{\frac{-\nu}{n}}.
\end{equation}
As in \cite{km1}, define
\begin{equation} \label{phi-21}
\phi_{n}(z) := \frac{1}{2} \int_{\beta_n}^z \frac{R_n(s)}{s} ds, \qquad z \in \mathbb{C} \setminus ( \Gamma_{0,n} \cup
\Gamma_{1,n} \cup [0, \infty) ),
\end{equation}
where the path of integration from $\beta_{n}$ to $z$ lies entirely in the region $\mathbb{C} \setminus ( \Gamma_{0,n} \cup
\Gamma_{1,n} \cup [0, \infty) )$, except for the initial point $\beta_{n}$. Similarly define
\begin{equation} \label{phi-22}
\tilde\phi_n(z) := \frac{1}{2} \int_{\bar\beta_n}^z \frac{R_n(s)}{s} ds, \qquad z \in \mathbb{C} \setminus (\Gamma_{0,n} \cup
\Gamma_{2,n} \cup [0, \infty)),
\end{equation}
where the path of integration from $\bar\beta_{n}$ to $z$ lies entirely in the region $\mathbb{C} \setminus ( \Gamma_{0,n} \cup
\Gamma_{2,n} \cup [0, \infty) )$, except for the initial point $\bar\beta_{n}$. It is not difficult to verify (also see
\cite{km1})
\begin{equation} \label{phi-phi2}
\phi_n(z) = \tilde\phi_n(z) \mp \pi i \qquad \textrm{for } z \in \Omega^{\pm}.
\end{equation}

To clarify contours $\Gamma_{0,n}$, $\Gamma_{1,n}$ and $\Gamma_{2,n}$ in (\ref{phi-21}) and (\ref{phi-22}), we consider
trajectories of the quadratic differential
\begin{equation}
- \frac{R_n(s)^2}{s^2} ds^2  = - \frac{(s-\beta_{n})(s- \bar\beta_{n})}{s^2} ds^2 ,
\end{equation}
which has two simple zeros at $\beta_{n}$ and $\bar\beta_{n}$ and a double pole at 0; see Strebel \cite{strebel}. In \cite{km1}
it is shown that there exist curves $\Gamma_{0,n}$, $\Gamma_{1,n}$ and $\Gamma_{2,n}$ as follows:
\begin{defn} \label{gamma01}

The contour $\Gamma_{0,n}$ is a curve from $\bar\beta_{n}$  to $\beta_{n}$ which crosses the negative real axis, so that
\begin{equation} \label{phi-21-prop}
\re \phi_{n,\pm}(z) = 0 \qquad \textrm{for} \quad z \in \Gamma_{0,n}.
\end{equation}
$\Gamma_{1,n}$ and $\Gamma_{2,n}$ are curves that form the analytic continuation of $\Gamma_{0,n}$ such that $\tilde\phi_n(z)$
and $\phi_n(z)$ are real and positive on them, respectively; see Figure~\ref{contoury}.

\end{defn}
\begin{figure}[h]
\centering
\includegraphics[width=200pt]{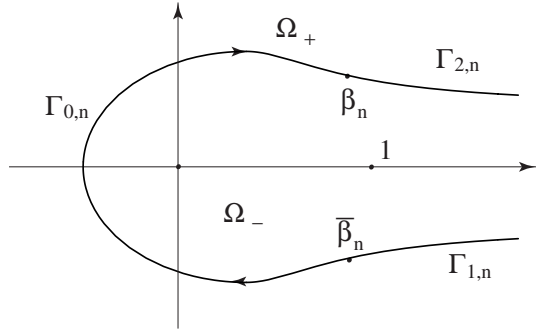}
\caption{Curves $\Gamma_{0,n}$, $\Gamma_{1,n}$ and $\Gamma_{2,n}$} \label{contoury}
\end{figure}

As in \cite{km1}, in this case we choose the curve $\Sigma$ in the RH problem for $Y$ to be $\Sigma = \Gamma_{0,n} \cup
\Gamma_{1,n} \cup \Gamma_{2,n}$. And we do not need a step $Y \mapsto U$ as in the Case I.

Following the similar analysis as in the proof of Lemma \ref{gamma0-szego}, we have:

\begin{lem}
As $n \to \infty$, the curve $\Gamma_{0,n}$ tends to the Szeg\H{o} curve $\mathcal{S}$ (\ref{szego-curve}).
\end{lem}

As an analog of (\ref{dmu-n}) in the Case I, we define a measure $d\mu_n$ to be
\begin{equation}
d \mu_n(y) = \frac{1}{2 \pi i} \frac{R_{n,+}(y)}{y} dy, \qquad y \in \Gamma_{0,n}.
\end{equation}
Using (\ref{phi-21}) and (\ref{phi-21-prop}), one can verify that $d \mu_n(y)$ is a probability measure on $\Gamma_{0,n}$; see
also \cite{km1}.


\subsection{The $g$ and $\phi$ functions}

As in the Case I, define the $g$\,-function to be
\begin{equation}
g_n(z) := \int_{\Gamma_{0,n} } \log(z-s) d \mu_n(s), \qquad z \in \mathbb{C} \setminus ( \Gamma_{0,n} \cup \Gamma_{1,n} ),
\end{equation}
where the logarithm $\log(z-s)$ is defined with a cut along $\Gamma_{0,n} \cup \Gamma_{1,n}$. Introduce
\begin{equation}
    g_{t,n}(z) := g_n(z) + (t-1) g_n^0(z)
\end{equation}
and
\begin{equation} \label{phi-tz2}
\phi_{t,n}(z) := \phi_n(z) + (t-1) \phi_n^0(z), \qquad \tilde\phi_{t,n}(z) := \tilde\phi_n(z) + (t-1) \phi_n^0(z),
\end{equation}
where
\begin{equation}
    g_n^0(z) := \frac{1}{2} (z - R_n(z)) \quad \hbox{and} \quad
    \phi_n^0(z) := \frac{R_n(z)}{2}, \qquad z \in \mathbb{C} \setminus \Gamma_{0,n}.
\end{equation}
Note that $g_n^0(z) = O(1/z)$ as $z \to \infty$.

Using similar analysis as in Propositions 3.7 and 3.8 in \cite{km1}, we get the following properties for $g_{t,n}$ of
$\phi_{t,n}$, which is an analog of Proposition \ref{g-prop} in the Case I.
\begin{prop} \label{g2-prop}

\begin{enumerate}

\item[(a)] There exists a constant $l_{t,n}$ such that
\begin{equation}
g_{t,n}(z) = \frac{1}{2} \left( A_n\log z + t z + l_{t,n} \right) - \phi_{t,n}(z), \quad z \in \mathbb{C} \setminus (
\Gamma_{0,n} \cup \Gamma_{1,n} \cup [0 , \infty) ).
\end{equation}
Here $\log z$ is defined with a cut along $[0, \infty)$. And the constant $l_{t,n}$ is explicitly given by
$$l_{t,n} =
2g_{t,n}(\beta_n) - (A_n \log \beta_n + t \beta_n).$$

\item[(b)] For $z \in \Sigma$, we have
\begin{equation}
g_{t,n,+}(z)-g_{t,n,-}(z) = \begin{cases} 2 \pi i , & z \in \Gamma_{1,n} ,\\
-2\phi_{t,n,+}(z) = 2 \phi_{t,n,-}(z), &  z \in  \Gamma_{0,n}.
\end{cases}
\end{equation}

\item[(c)] We also have
\begin{equation}
g_{t,n,+}(z) + g_{t,n,-}(z) = \begin{cases} A_n \log{z}+ t z + l_{t,n} , & z \in \Gamma_{0,n}\text{,}\\
A_n \log{z} + t z + l_{t,n} -2\tilde\phi_{t,n}(z) \text{,}& z \in \Gamma_{1,n}, \\
A_n \log{z} + t z + l_{t,n} -2\phi_{t,n}(z)\text{,} & z \in \Gamma_{2,n} \text{.}
\end{cases}
\end{equation}

\end{enumerate}

\end{prop}


\subsection{First and second transformations $Y \mapsto T$, $T\mapsto S$}

To normalize the RH problem at infinity, introduce the first transformation $Y \mapsto T$ to be
\begin{equation}
T(z) = e^{-\frac{1}{2} n l_{t,n} \sigma_3} Y(z) \, e^{ - n g_{t,n}(z)  \sigma_3} \, e^{\frac{1}{2} n l_{t,n} \sigma_3} .
\end{equation}
From the RH problem for $Y$ in Section \ref{rh-op} and Proposition \ref{g2-prop}, we obtain a RH problem for $T$ as follows.
\begin{enumerate}

\item[(a)] $T(z)$ is analytic for $z \in \mathbb{C} \setminus \Sigma,$ $\Sigma = \Gamma_{0,n} \cup \Gamma_{1,n} \cup \Gamma_{2,n}$;

\item[(b)]
\begin{align}
T_+(z) = \ & T_-(z) \left( \begin{matrix} e^{2 n \phi_{t,n,+}(z)} & (z-1)^{2b} \\ 0 & e^{2 n \phi_{t,n,-}(z)}
\end{matrix} \right) & \textrm{for } z \in \Gamma_{0,n}; \label{t2-jump0} \\
T_+(z) = \ & T_-(z) \left( \begin{matrix} 1 & e^{-2n \tilde\phi_{t,n}(z)} (z-1)^{2b} \\ 0 & 1
\end{matrix} \right) & \textrm{for } z \in \Gamma_{1,n}; \\
T_+(z) = \ & T_-(z) \left( \begin{matrix} 1 & e^{-2n \phi_{t,n}(z)} (z-1)^{2b} \\ 0 & 1
\end{matrix} \right) & \textrm{for } z \in \Gamma_{2,n};
\end{align}

\item[(c)] as $z \to \infty$
\begin{equation*}
T(z) = I + O \left( \frac{1}{z} \right), \qquad z \in \mathbb{C} \setminus \Sigma.
\end{equation*}

\end{enumerate}

From (\ref{phi-21-prop}) and (\ref{phi-tz2}), we know that, when $|t-1|$ is small, the boundary values of $\re \phi_{t,n}(z)$ on
both sides of $\Gamma_{0,n}$ is close to 0. This means the jump matrix for $T$ on $\Gamma_{0,n}$ in (\ref{t2-jump0}) has
oscillatory diagonal entries. Similarly as in the previous case, to remove this oscillating jump, we introduce the second
transformation and open the lens as follows. Let $\Gamma_{4,n}$ be a smooth curve connecting $\beta_n$, $\bar\beta_n$ and 1.
Note that since $\im \beta_n \neq \arg \beta_n$, we can not choose $\Gamma_{4,n}$ to be the curve $\im z = \arg z$. However, as
$\beta_{n} $ is close to 1 for $n$ large, we can choose $\Gamma_{4,n}$ such that it is close to the curve $\im z = \arg z$ for
$n$ large. Then choose $\Gamma_{3,n}$ be the continuation of $\Gamma_{4,n}$ such that $\Gamma_{3,n} \cup \Gamma_{4,n}$ becomes a
closed contour and contains $\Gamma_{0,n}$ in its interior. Then the curves $\Gamma_{j,n}, j=0,1, \cdots, 4,$ partition the
complex plane into domains $\Omega_\infty$ and $\Omega_{k}, k = 0,1, 2,3 $; see Figure~\ref{contours}.
\begin{figure}[h]
\centering
\includegraphics[width=220pt]{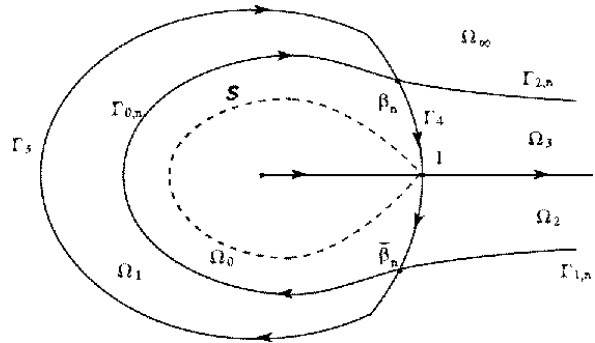}
\caption{The contour for the RH problem for $T$} \label{contours}
\end{figure}
Define
\begin{align}
S(z) & =  T(z) & \textrm{for } z \in \Omega_\infty  \\
S(z) & =  T(z) \left( \begin{matrix} 0 & (z-1)^{2b} \\ -  (z-1)^{-2b} & e^{2 n \phi_{t,n}(z)} \end{matrix} \right)
& \textrm{for } z \in \Omega_0; \\
S(z) & =  T(z) \left( \begin{matrix} 1 & 0 \\ -  (z-1)^{-2b} e^{2n \phi_{t,n}(z) } & 1 \end{matrix} \right) &
\textrm{for } z \in \Omega_1; \\
S(z) & =  T(z) \left( \begin{matrix} 1 & (z-1)^{2b} e^{-2n \tilde\phi_{t,n}(z) } \\ 0 & 1 \end{matrix} \right) & \textrm{for } z
\in \Omega_2; \\
 S(z) & =  T(z) \left( \begin{matrix} 1 & (z-1)^{2b} e^{-2n \phi_{t,n}(z) } \\ 0 & 1 \end{matrix} \right) &
\textrm{for } z \in \Omega_3.
\end{align}
Direct calculation shows that $S_+(z) = S_-(z)$ for $z \in \Gamma_{0,n} \cup \Gamma_{1,n} \cup \Gamma_{2,n}$. Therefore, $S(z)$
has an analytic continuation across $\Gamma_{0,n} \cup \Gamma_{1,n} \cup \Gamma_{2,n}$ and we obtain a RH problem for $S$ on the
contour $\Sigma^S = \Sigma_{3,n} \cup \Sigma_{4,n} \cup [0, \infty)$. Here $\Sigma^S$ divides the complex plane into two domains
$\Omega_0^S$ and $\Omega_\infty^S$; see Figure~\ref{contourt22}. As before $(z-1)^{2b}$ is defined with a cut along $[1,
\infty)$.

\begin{figure}[h]
\centering
\includegraphics[width=220pt]{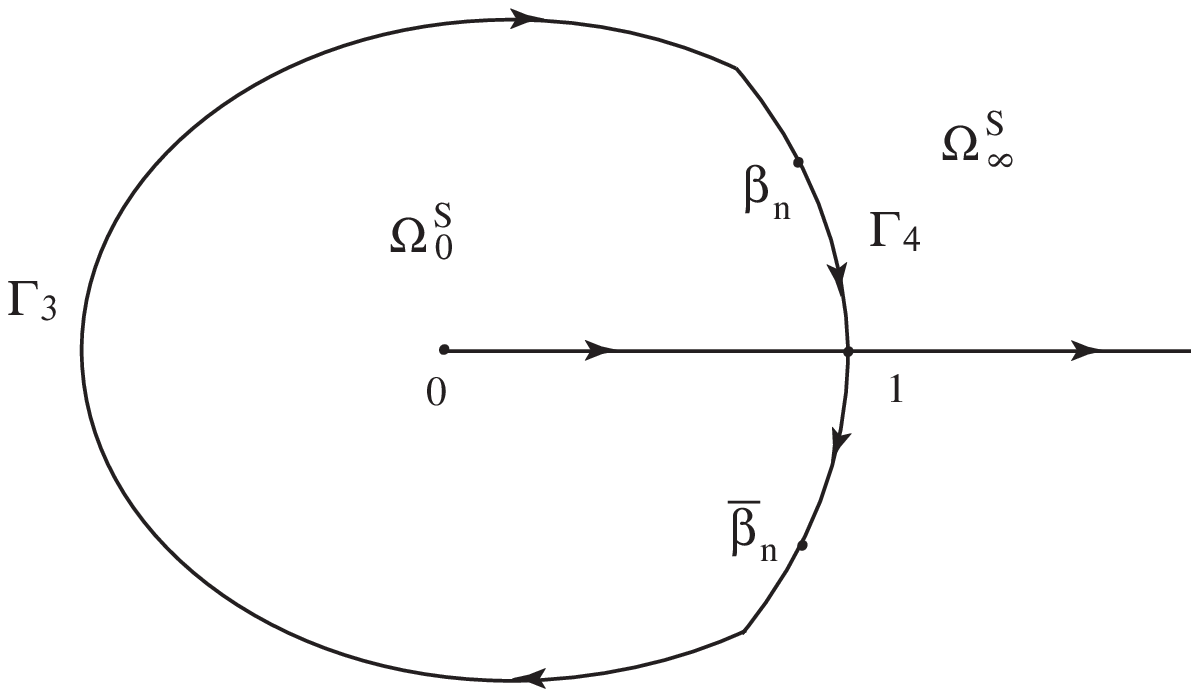}
\caption{The contour $\Sigma^S$ for the RH problem for $S$} \label{contourt22}
\end{figure}

\begin{enumerate}

\item[(a)] $S$ is analytic for $z \in \mathbb{C} \setminus
\Sigma^S$, see Figure~\ref{contourt22};

\item[(b)]
\begin{align}
S_+(z) = \ & S_-(z) \left( \begin{matrix} 1 & 0 \\  (z-1)^{- 2b} e^{2 n \phi_{t,n}(z) } & 1 \end{matrix} \right) &
\textrm{for } z \in \Gamma_{3,n}, \\
S_+(z) = \ & S_-(z) \left( \begin{matrix} 1 & \displaystyle\frac{(z-1)^{2b}}{e^{2n \phi_{t,n,+}(z)}}   - \frac{(z-1)^{2b}}{e^{2n
\phi_{t,n,-}(z)} )} \\ 0 & 1
\end{matrix} \right) & \textrm{for } z \in
(0, 1), \\
S_+(z) = \ & S_-(z) \left( \begin{matrix} e^{2n \phi_{t,n}(z)} & 0 \\ (z-1)^{- 2b} & e^{- 2n \phi_{t,n}(z)} \end{matrix} \right)
& \textrm{for } z \in
\Gamma_{4,n},\\
S_+(z) = \ & S_-(z) \left( \begin{matrix} 1 & \displaystyle\frac{(z-1)_+^{2b}}{e^{2n \phi_{t,n,+}(z)}}
-  \frac{ (z-1)_-^{2b}}{e^{2n \phi_{t,n,-}(z)} )}   \\
0 & 1 \end{matrix} \right) & \textrm{for } z \in (1, \infty);
\end{align}

\item[(c)] as $z \to \infty$
\begin{equation*}
S(z) = I + O\left(\frac{1}{z}\right);
\end{equation*}

\item[(d)] as $z \to 1$
\begin{eqnarray}
S(z) \left( \begin{matrix} e^{2n \phi_{t,n,\pm}(1)} & - (z-1)^{2b} \\  (z-1)^{-2b} & 0 \end{matrix} \right) & = & O (1) \quad
\textrm{for }  z \in \Omega_0^S \cap \mathbb{C}^\pm,  \\
S(z) \left( \begin{matrix} 1 & - e^{-2n \phi_{t,n,+}(1)} (z-1)^{2b} \\ 0 & 1 \end{matrix} \right)
& = & O (1) \quad \textrm{for }  z \in \Omega_\infty^S \cap \mathbb{C}^+,  \\
S(z) \left( \begin{matrix} 1 & - e^{-2n \tilde\phi_{t,n,-}(1)} (z-1)^{2b} \\ 0 & 1 \end{matrix} \right) & = & O (1) \quad
\textrm{for }  z \in \Omega_\infty^S \cap \mathbb{C}^-.
\end{eqnarray}

\end{enumerate}

Like in the Case I, condition (d) is important for the above RH problem for $S$. It makes sure that we have a unique and desired
solution for $S$.

\subsection{Definition of $\psi$ functions}

For future analysis, it is convenient to introduce new branch cuts for the functions $(z-1)^{2b}$, $R_n(z)$, $\phi_{t,n}(z)$ and
$\tilde\phi_{t,n}(z)$. More precisely, we want the branch cuts taken on $\Gamma_{4,n}$. Let $\Gamma_{j\, ,+}$ and $\Gamma_{j\,
,-}$ denote $\Gamma_j \cap \mathbb{C}^+$ and $\Gamma_j \cap \mathbb{C}^-$, respectively.

First for $(z-1)^{2b}$, instead of taking the cut along $[1,\infty)$, we take new cut along $\Gamma_{4,n,-} \cup \Gamma_{3,n,-}
\cup (-\infty, x_0]$, where $x_0$ is the intersection point of $\Gamma_{3,n}$ and $(-\infty, 0]$; see Figure~\ref{cutz-1}. Thus,
\begin{equation} \label{z-1-branch2}
(z-1) \mapsto \begin{cases} (z-1) e^{2 \pi i} & \textrm{for } z \in \Omega_\infty^S \cap \mathbb{C}^-, \\ (z-1) &
\hbox{elsewhere}.
\end{cases}
\end{equation}
\begin{figure}[h]
\centering
\includegraphics[width=200pt]{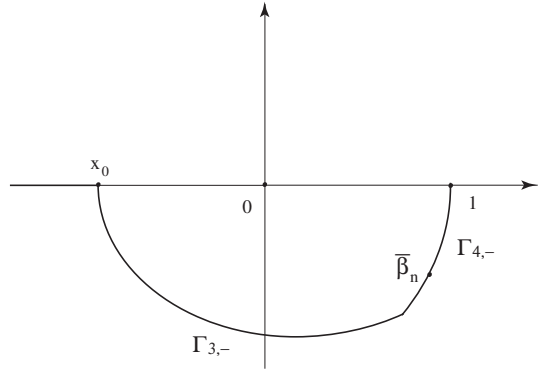}
\caption{The new cut of $(z-1)^{2b}$} \label{cutz-1}
\end{figure}

Then, define
\begin{equation}
\widetilde{R}_n(z) := \sqrt{(z - \beta_n)(z - \bar\beta_n)}, \qquad z \in \mathbb{C} \setminus \Gamma_{4,n},
\end{equation}
where $\widetilde{R}_n(z)$ behaves like $z$ as $z \to \infty$. Also introduce
\begin{figure}[h]
\centering
\includegraphics[width=180pt]{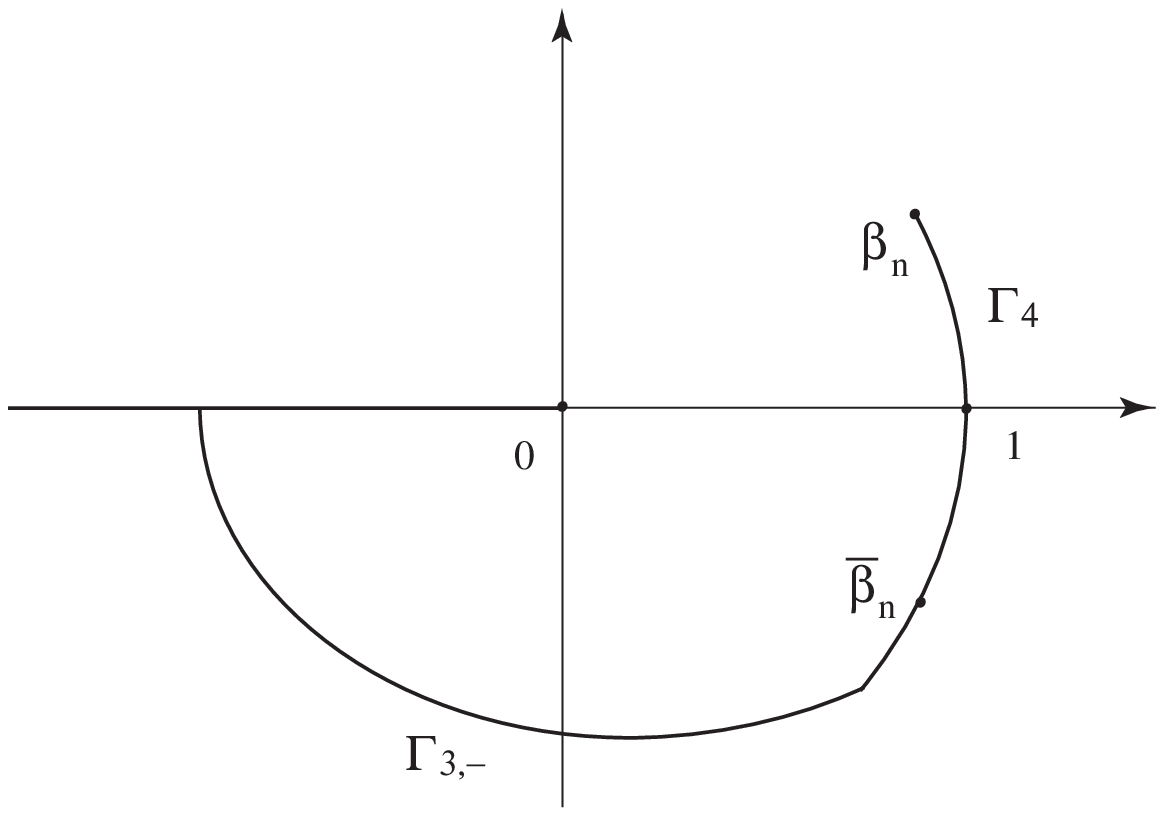}
\hspace{20pt}
\includegraphics[width=180pt]{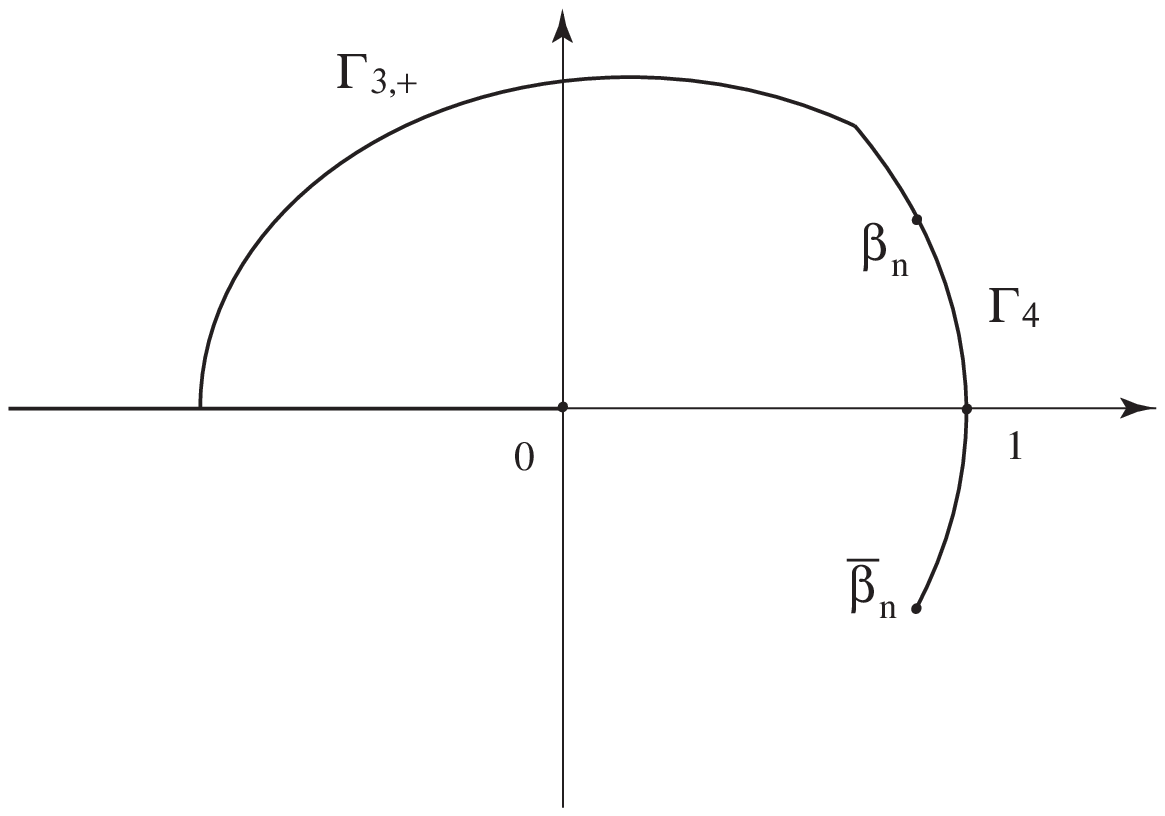}
\caption{The cuts of $\psi_{t,n}$ and $\tilde\psi_{t,n}$} \label{psi-cut}
\end{figure}
\begin{equation}  \label{psi-tn}
\psi_{t,n}(z) := \psi_n(z) + (t-1) \psi_n^0(z) \quad \hbox{and} \quad \tilde\psi_{t,n}(z) := \tilde\psi_n(z) + (t-1) \psi_n^0(z)
\end{equation}
where
\begin{align}
\psi_n(z) := & \frac{1}{2} \int_{\beta_n}^z \frac{\widetilde{R}_n(s)}{s} ds , \qquad z \in \mathbb{C} \setminus (\Gamma_{3,n,-}
\cup \Gamma_{4,n} \cup (-\infty, 0]),
\label{psi-n-1} \\
\tilde\psi_n(z)  := & \frac{1}{2} \int_{\bar\beta_n}^z \frac{\widetilde{R}_n(s)}{s} ds , \qquad z \in \mathbb{C} \setminus
(\Gamma_{3,n,+} \cup \Gamma_{4,n} \cup (-\infty, 0])
\end{align}
and
\begin{equation} \label{psi-n-0}
\psi_n^0(z) := \frac{\widetilde{R}_n(z)}{2}, \qquad z \in \mathbb{C} \setminus \Gamma_{4,n}.
\end{equation}
For illustration of the branch cuts of $\psi_{t,n}$ and $\tilde\psi_{t,n}$, see Figure~\ref{psi-cut}. Here are some useful
properties of these new auxiliary functions, which can be obtained directly from Proposition 4.6.1 in \cite{dev}.

\begin{prop} \label{newfunctions}

\begin{itemize}

\item[(a)] The connection between $R_n$ and $\widetilde{R}_n$:
\begin{equation}
\widetilde{R}_n(z) = \begin{cases} -R_n(z), & \textrm{for } z \in \Omega_0, \\
R_n(z), & \textrm{elsewhere}.
\end{cases}
\end{equation}

\item[(b)] The connection between $\psi_{t,n}$ and $\phi_{t,n}$:
\begin{equation}
\psi_{t,n}(z) = \begin{cases} \phi_{t,n}(z), & \text{ for }
z \in  ( \Omega_3 \cup \Omega_1 \cup \Omega_\infty ) \cap \mathbb{C}^+ , \\
-\phi_{t,n}(z), & \text{ for } z \in \Omega_0 \cap \mathbb{C}^+, \\
-\phi_{t,n}(z) + A_n \pi i, & \text{ for } z \in \Omega_0 \cap \mathbb{C}^-, \\
\phi_{t,n}(z) - A_n \pi i, & \text{ for } z \in \Omega_2, \\
\phi_{t,n}(z) + A_n \pi i, & \text{ for } z \in  \Omega_1 \cap \mathbb{C}^- , \\
\phi_{t,n}(z) + (2- A_n) \pi i, & \text{ for } z \in  \Omega_\infty \cap \mathbb{C}^-. \end{cases}
\end{equation}

\item[(c)] The connection between $\psi_{t,n}$ and $\tilde\psi_{t,n}$:
\begin{equation}
\psi_{t,n}(z)  = \begin{cases} \tilde\psi_{t,n}(z) - \frac{\nu}{n} \pi i , & \text{ for } z \in \Omega_0^S, \\
\tilde\psi_{t,n}(z) + \frac{\nu}{n} \pi i, & \text{ for } z \in \Omega_\infty^S. \end{cases}
\end{equation}

\end{itemize}

\end{prop}

Then, with the help of (\ref{z-1-branch2}) and Proposition \ref{newfunctions}, the RH problem for $S$ can be rewritten as
follows. Now $(z-1)^{2b}$ is defined with a cut along $\Gamma_{4,n,-} \cup \Gamma_{3,n,-} \cup (-\infty, x_0]$.
\begin{enumerate}

\item[(a)] $S$ is analytic for $z \in \mathbb{C} \setminus
\Sigma^S$; see Figure~\ref{contourt22};

\item[(b)]
\begin{align} \label{s2-jump1}
S_+(z) = & \ S_-(z) \left( \begin{matrix} 1 & 0 \\  (z-1)^{- 2b} e^{2 n \psi_{t,n}(z) } & 1 \end{matrix} \right), &
z \in \Gamma_{3,n,+}, \\
S_+(z) = & \ S_-(z) \left( \begin{matrix} 1 & 0 \\ (z-1)_-^{- 2b}  e^{2 n \tilde\psi_{t,n}(z) } & 1 \end{matrix}
\right),&  z \in \Gamma_{3,n,-}, \\
S_+(z) = & \ S_-(z) \left( \begin{matrix} e^{2n \psi_{t,n,+}(z)} & 0 \\ (z-1)^{- 2b} & e^{2n \psi_{t,n,-}(z)}
\end{matrix} \right), &
z \in \Gamma_{4,n,+}, \\
S_+(z) = & \ S_-(z) \left( \begin{matrix} e^{2n \tilde\psi_{t,n,+}(z)} & 0 \\ (z-1)_-^{- 2b} & e^{2n \tilde\psi_{t,n,-}(z)}
\end{matrix} \right), &  z \in \Gamma_{4,n,-}, \\
S_+(z) = & \ S_-(z) \left( \begin{matrix} 1 & (z-1)^{2b} e^{-2n \psi_{t,n}(z)} (e^{-2 \nu \pi i} - 1) \\ 0 & 1
\end{matrix} \right), &  z \in (0, 1),
\end{align}
\begin{equation}
S_+(z) = S_-(z) \left( \begin{matrix} 1 & (z-1)^{2b} e^{-2n \psi_{t,n}(z)} (1 - e^{ 2 (\nu + 2b) \pi i}) \\ 0 & 1
\end{matrix} \right), \qquad z \in (1, \infty), \label{s2-jump6}
\end{equation}

\item[(c)] as $z \to \infty$
\begin{equation*}
S(z) = I + O\left(\frac{1}{z}\right);
\end{equation*}

\item[(d)]  as $z \to 1$
\begin{eqnarray}
S(z) \left( \begin{matrix} e^{2n \phi_{t,n,\pm}(1)} & - (z-1)^{2b} \\  (z-1)^{-2b} & 0 \end{matrix} \right) & = & O (1) \quad
\textrm{for }  z \in \Omega_0^S \cap \mathbb{C}^\pm, \label{s-1-behavior1} \\
S(z) \left( \begin{matrix} 1 & - e^{-2n \phi_{t,n,+}(1)} (z-1)^{2b} \\ 0 & 1 \end{matrix} \right)
& = & O (1) \quad \textrm{for }  z \in \Omega_\infty^S \cap \mathbb{C}^+, \label{s-1-behavior2} \\
S(z) \left( \begin{matrix} 1 & -  e^{-2n \tilde\phi_{t,n,-}(1) + 4b \pi i} (z-1)^{2b} \\ 0 & 1 \end{matrix} \right) & = & O (1)
\quad \textrm{for }  z \in \Omega_\infty^S \cap \mathbb{C}^-.  \label{s-1-behavior3}
\end{eqnarray}
\end{enumerate}

Following similar analysis as in the Case I, one can see that most of the jump matrices in the above RH problem tend to the
identity matrix as $n \rightarrow \infty$ at an exponential rate. More precisely, we have the following proposition.

\begin{prop} \label{b-delta2}

Let $\delta$ be a fixed small constant and $B_\delta : = \{ z \in \mathbb{C} \ | \ |z-1|< \delta \}$. Then there exist $\eta>0$
and $\varepsilon > 0$ such that for all $t \in \mathbb{R}$ with $|t-1|<\eta$ and for all $n$ large enough, we have
$$\re
\psi_{t,n}(z) < - \varepsilon \qquad \hbox{for } z \in (\Gamma_{3,n} \cup \Gamma_{4,n}) \setminus B_\delta$$ and
$$\re \psi_{t,n}(z) > \varepsilon ( |z| + 1) \qquad \hbox{for } z \in
(0,\infty) \setminus B_\delta.$$ The same inequalities hold for $\re\tilde\psi_{t,n}(z)$.

\end{prop}
\begin{proof}
This proposition is similar with Proposition \ref{b-delta} and the proof is similar, too.
\end{proof}


\subsection{Construction of the local parametrix}

From Proposition \ref{b-delta2}, we know  that the jump matrices for $S$ are exponentially close to the identity matrix as $n
\to \infty$, except for the ones in the neighborhood of 1. Let us focus on the RH problem of $S$ restricted to a neighborhood
$B_\delta $ of $z=1$.  We seek a $2 \times 2$ matrix valued function $P$ that satisfies the following RH problem.

\begin{enumerate}

\item[(a)] $P(z)$ is analytic for $z \in B_\delta \setminus
\Sigma^S$, and continuous on $\overline{B}_\delta \setminus \Sigma^S$.

\item[(b)] $P(z)$ satisfies the same jump conditions on $\Sigma^S
\cap B_\delta$ as $S$ does; see (\ref{s2-jump1})--(\ref{s2-jump6}).

\item[(c)] on $\partial B_\delta$, as $n \to \infty$
\begin{equation} \label{p2-0condition}
P(z) = \left( I + O \left( \frac{1}{\sqrt{n}} \right) \right) n^{\frac{b}{2} \sigma_3} \qquad \textrm{for } z \in \partial
B_\delta \setminus \Sigma^S.
\end{equation}

\item[(d)] $P(z)$ satisfies the same local behavior near 1 as $S$ does; see (\ref{s-1-behavior1})--(\ref{s-1-behavior3}).

\end{enumerate}
As in the previous case, we can not avoid the factor $n^{\frac{b}{2} \sigma_3}$ in (\ref{p2-0condition}).

To transform the RH problem for $P$ into a RH problem for $Q$ with constant jump matrices, we seek $P$ in the following form
\begin{equation} \label{pq}
P(z) = \begin{cases} \sigma_1 Q(z) \sigma_1 e^{(n \psi_{t,n}(z) - \frac{1}{2}(\nu + b)\pi i) \sigma_3} (z-1)^{-b \sigma_3}
\qquad \textrm{for } z \in B_\delta \cap \Omega^S_0, \\
\sigma_1 Q(z) \sigma_1 e^{(n \tilde\psi_{t,n}(z) + \frac{1}{2}(\nu - b)\pi i) \sigma_3} (z-1)^{-b \sigma_3} \qquad \textrm{for }
z \in B_\delta \cap \Omega^S_\infty,
\end{cases}
\end{equation}
where $(z-1)^b$ is defined with a cut along $\Gamma_{4,n,-} \cup \Gamma_{3,n,-} \cup(-\infty, x_0]$. The factors $e^{n
\psi_{t,n}(z) \sigma_3 } (z-1)^{-b \sigma_3}$ and $e^{n \tilde\psi_{t,n}(z) \sigma_3} (z-1)^{-b \sigma_3}$ in (\ref{pq}) are
introduced to cancel the corresponding factors in the jump matrices in (\ref{s2-jump1})--(\ref{s2-jump6}) and make them
independent of $z$.

Reorient the curves $\Sigma^S \cap B_\delta$ and let them all extend from the point 1. Then it can be verified that, in order
for $P(z)$ to satisfy the desired RH problem, $Q(z)$ should satisfy a RH problem as follows:
\begin{enumerate}

\item[(a)] $Q(z)$ is analytic for $z \in B_\delta \setminus
\Sigma^S$, and continuous on $\overline{B}_\delta \setminus \Sigma^S$;

\item[(b)]
\begin{align}
Q_+(z) & =  Q_-(z) \left( \begin{matrix} 1 & 0 \\ e^{-(\nu + b) \pi i} - e^{ (\nu + 3b ) \pi i} & 1
\end{matrix} \right), \hspace{65pt} z \in (1, 1+ \delta) ; \label{q2-jump1} \\
Q_+(z) & =  Q_-(z) \left( \begin{matrix} 1 & -e^{(\nu + b)\pi i}  \\ 0 & 1 \end{matrix} \right), \hspace{55pt}
z \in (\Gamma_{3,n,+} \cup \Gamma_{4,n,+})\cap B_\delta; \label{q2-jump2} \\
Q_+(z) & =  Q_-(z) \left( \begin{matrix} 1 & 0 \\ (e^{\nu \pi i} - e^{-\nu \pi i}) e^{-(2\nu + b) \pi i}  & 1
\end{matrix} \right), \hspace{45pt} z \in (1 - \delta,1); \label{q2-jump3} \\
Q_+(z) & =  Q_-(z) \left( \begin{matrix} 1 & e^{(3 \nu + b)\pi i}  \\ 0 & 1
\end{matrix} \right) e^{2 (\nu + b) \pi i \sigma_3}, \quad z \in (\Gamma_{3,n,-} \cup \Gamma_{4,n,-}) \cap B_\delta;
\label{q2-jump4}
\end{align}

\item[(c)] for $z \in \partial B_\delta$ as $n \to \infty$,
\begin{align}
    Q(z) & = \left( I + O\left(\frac{1}{\sqrt{n}}\right) \right) (\sqrt{n}(z-1))^{-(\nu + b) \sigma_3}
    \exp\biggl[\frac{n}{2} (z-1- \log{z}) \nonumber \\
    & \hspace{15pt} + \frac{n^{\frac{1}{2}} \sqrt{2} \, L (z-1)}{2}
    + \frac{\nu}{2} (\log{z} + \log(-\nu) - 1) - \frac{b}{2} \pi i  \biggr] \sigma_3 \label{q2-match}
\end{align}
where $\log{z}$ and $(z-1)^{-(\nu + b)}$ are defined with cuts along $(-\infty, 0]$ and $\Gamma_{4,n,-} \cup \Gamma_{3,n,-}
\cup(-\infty, x_0]$, respectively;

\item[(d)] $Q(z)$ has the following behavior as $z \to 1$:
\begin{align}
Q(z) \left( \begin{matrix} 0 & e^{ (\nu + b) \pi i}  \\
- e^{ -(\nu + b) \pi i}  & 1 \end{matrix} \right) (z-1)^{ b \sigma_3} & = O(1), & z \in \Omega_0^S \cap
\mathbb{C}^+, \label{q2-0-behave1} \\
Q(z) \left( \begin{matrix} 0 & e^{ (\nu + b) \pi i}  \\
- e^{ -(\nu + b) \pi i}  & e^{- 2 \nu \pi i} \end{matrix} \right) (z-1)^{ b \sigma_3} & = O(1), &  z \in \Omega_0^S \cap
\mathbb{C}^-, \\
Q(z) \left( \begin{matrix} 1 & 0 \\
- e^{ -(\nu + b) \pi i} & 1 \end{matrix} \right) (z-1)^{ b \sigma_3} & = O(1), & z \in \Omega_\infty^S \cap
\mathbb{C}^+,  \\
Q(z) \left( \begin{matrix} 1 & 0 \\
- e^{ (\nu + 3b) \pi i}  & 1 \end{matrix} \right) (z-1)^{ b \sigma_3} & = O(1), &  z \in \Omega_\infty^S \cap \mathbb{C}^-,
\label{q2-0-behave4}
\end{align}
where $(z-1)^b$ is defined with a cut along $\Gamma_{4,n,-} \cup \Gamma_{3,n,-} \cup(-\infty, x_0]$.

\end{enumerate}

The above RH problem for $Q$ is very similar to the RH problem in Section \ref{rhp-q}. One can see that there is only a constant
difference between corresponding formulas in these two RH problems. Also again there is no difference for the jump matrices
(\ref{q2-jump2}) and (\ref{q2-jump4}) in the neighborhood of $\beta_n$ and $\bar\beta_n$, respectively.

To obtain the limiting behavior as $z \to 1$ in (\ref{q2-0-behave1})--(\ref{q2-0-behave4}), one have to make use of
(\ref{phi-phi2}) and Proposition \ref{newfunctions}. To derive the matching condition in (\ref{q2-match}), we need the following
asymptotic formulas for $\psi_n^0(z)$ and $\psi_n(z)$ as $n \to \infty$. For $z \in \partial B_\delta$, we have from
(\ref{psi-n-0})
\begin{equation} \label{psi0a}
\psi_n^0(z) = \frac{1}{2}(z-1) - \frac{\nu (z+1)}{2n(z-1)} + O \left( \frac{1}{n^2} \right)
\end{equation}
and from (\ref{psi-n-1}), (calculated with Maple)
\begin{align}
\psi_{1,n}(z) & = \frac{1}{2}(z - 1 -\log{z}) - \frac{\nu}{2n} \log{n} + \frac{\nu}{2n}
\Big( -2 \log(z-1) + \pi i + \log{z} + \log(-\nu) - 1 \Big) \nonumber \\
& \quad + \frac{\nu^2}{n^2 (z-1)^2} + O \left( \frac{1}{n^3} \right), \label{psi1a}
\end{align}
where $\log{z}$ and $\log(z-1)$ are defined with cuts along $(-\infty,0]$ and $\Gamma_{4,n,-} \cup \Gamma_{3,n,-} \cup(-\infty,
x_0]$, respectively; see \cite{dev}. The rest of the calculations is the same as in the Case I and we do not go into details.

Now we are in a situation very similar to the Case I in Section \ref{case1-construct1}. First, from the factors $(z-1)^{b
\sigma_3}$ in (\ref{q2-0-behave1})--(\ref{q2-0-behave4}) and $(z-1)^{- (\nu + b) \sigma_3} $ in (\ref{q2-match}), again we have
\begin{equation}
\Theta = - b, \qquad \Theta_{\infty} = \nu + b,
\end{equation}
which are the same as in (\ref{theta-def}). Next, from (\ref{q2-jump1})--(\ref{q2-jump4}) we need the Stokes multipliers
\begin{equation} \label{case2-sm}
\begin{array}{ll}
s_1= e^{-(\nu + b) \pi i} - e^{ (\nu + 3b ) \pi i} , & s_2 = -e^{(\nu + b)\pi i}, \\
s_3 = (e^{\nu \pi i} - e^{-\nu \pi i}) e^{-(2\nu + b) \pi i} , & s_4 = e^{(3 \nu + b)\pi i}.
\end{array}
\end{equation}
Although (\ref{case1-sm}) and (\ref{case2-sm}) are not the same, the above Stokes multipliers, like those in (\ref{case1-sm}),
are again equivalent to those in (\ref{case0-sm}) under the transformation in (\ref{sm-trans}), but with a different constant $d
= - e^{- \nu \pi i } $. Then following almost the same analysis as in the construction of $Q$ in (\ref{p1con}), except for a
constant difference, we construct our parametrix by
\begin{equation} \label{q-para}
Q(z) = E(z) \Psi\left(n^{\frac{1}{2}} f(z), L \frac{z-1}{\sqrt{2} f(z)} \right)
\end{equation}
with the same $f(z)$ as given in (\ref{f-def}) and
\begin{equation}
E(z) = \left( \frac{f(z)}{z-1} \right)^{(\nu  + b) \sigma_3} ( - \nu e^{-1} \; z)^{\frac{\nu}{2} \sigma_3} e^{- \frac{b}{2} \pi
i \sigma_3}.
\end{equation}

There is one thing remaining about the curve $ (\Gamma_{3,n} \cup \Gamma_{4,n}) \cap B_\delta$ under the mapping of $f(z)$.
Observe that $f(z)$ does not map $ (\Gamma_{3,n} \cup \Gamma_{4,n}) \cap B_\delta$ to the imaginary axis. However, for $n$ large
enough, according to our choice of $\Gamma_{3,n}$ and $\Gamma_{4,n}$, $f((\Gamma_{3,n} \cup \Gamma_{4,n}) \cap B_\delta)$ is
close to the imaginary axis. In the RH problem for $\Psi$ in Section \ref{rh-piv}, by the principle of analytic continuation, we
can modify $i \, \mathbb{R}$ to a new contour $\Lambda_n$ such that $f((\Gamma_{3,n} \cup \Gamma_{4,n}) \cap B_\delta) \subset
\Lambda_n$ for $n$ large enough. Therefore, we can still use the parametrix constructed in (\ref{q-para}).


\subsection{Proof of Theorem \ref{thm1} in Case II}

Having $Q$ as given in (\ref{q-para}), we can continue the analysis as in Sections \ref{r-section}--\ref{thm1-proof}. The rest
of the proof is the same as in the Case I and we do not go into details.


\section{Proof of Theorem \ref{thm3}} \label{thm3-proof}

Then we are ready to prove Theorem \ref{thm3}.

\begin{proof}
Let us consider Case I first. From (\ref{y-sol}), (\ref{yy*}) and (\ref{t-y}), we have
\begin{equation} \label{pi&t}
\pi_n(z) = (Y(z))_{11} = (U(z))_{11} = (T(z))_{11} e^{n g_{t,n}(z)}.
\end{equation}
First for $z \in \Omega_\infty^S$, we get from (\ref{t-s3})
\begin{equation} \label{pi-region2}
\pi_n(z) = (S(z))_{11}  e^{n g_{t,n}(z)}, \qquad  z \in \Omega_\infty^S.
\end{equation}
By (\ref{r-s}) and (\ref{r-nasy}), it is known that
\begin{equation} \label{s&r}
S(z) = n^{-\frac{b}{2} \sigma_3} R(z) n^{\frac{b}{2} \sigma_3} = n^{-\frac{b}{2} \sigma_3} \left( I + O \left(
\frac{1}{\sqrt{n}} \right) \right) n^{\frac{b}{2} \sigma_3} \quad \textrm{uniformly for } z \in \mathbb{C} \setminus B_\delta.
\end{equation}
Combining (\ref{pi-region2}) and (\ref{s&r}) gives us
\begin{equation} \label{pi-zero1}
\pi_n(z)= \left( 1 + O \left( \frac{1}{\sqrt{n}} \right) \right) e^{n g_{t,n}(z)} \quad \textrm{uniformly for } z \in
\Omega_\infty^S \setminus B_\delta.
\end{equation}
So, for $n$ large enough, there is no zero of $\pi_n(z)$ for $z \in \Omega_\infty^S \setminus B_\delta$.

Then let us consider $z \in \Omega_0^S.$ From (\ref{phi-phi}) and (\ref{t-s1}) we have
\begin{equation} \label{t&s1}
(T(z))_{11} = (-1)^n e^{-2n \phi_{t,n}(z) \mp \nu \pi i} (S(z))_{11} + (-1)^n \frac{e^{- \nu \pi i}- e^{\nu \pi i}}{ (z-1)^{2b}}
(S(z))_{12} \quad \textrm{for } z \in \Omega_0 \cap \mathbb{C}^{\pm},
\end{equation}
and from (\ref{phi-phi}) and (\ref{t-s2}) we get
\begin{equation} \label{t&s2}
(T(z))_{11} = (S(z))_{11} + \frac{e^{- \nu \pi i}- e^{\nu \pi i}}{(z-1)^{2b}} e^{2n \phi_{t,n}(z) \pm \nu \pi i} (S(z))_{12}
\qquad \textrm{for } z \in \Omega_1 \cap \mathbb{C}^{\pm}.
\end{equation}
Then, combining (\ref{pi&t}), (\ref{s&r}), (\ref{t&s1}) and (\ref{t&s2}) gives us
\begin{align}
\pi_n(z) = & \ e^{n g_{t,n}(z)} \left[ (-1)^n e^{-2n \phi_{t,n}(z) \mp \nu \pi i} \left(1 + O \left( \frac{1}{\sqrt{n}} \right)
\right) \right. \nonumber \\
& \left. \hspace{50pt} + (-1)^n \frac{e^{- \nu \pi i}- e^{\nu \pi i}}{ n^b (z-1)^{2b}} (R(z))_{12} \right]  \quad \textrm{for }
z \in \Omega_0 \cap \mathbb{C}^{\pm} \setminus B_\delta , \label{pi-region0}
\end{align}
and
\begin{align}
\pi_n(z) = & \ e^{n g_{t,n}(z)} \left[ 1 + O \left( \frac{1}{\sqrt{n}} \right) + \frac{e^{- \nu \pi i}- e^{\nu \pi i}}{n^b
(z-1)^{2b}} e^{2n
\phi_{t,n}(z) \pm \nu \pi i} O \left( \frac{1}{\sqrt{n}} \right) \right] \nonumber \\
& \textrm{\hspace{220pt} for } z \in \Omega_1 \cap \mathbb{C}^{\pm} \setminus B_\delta . \label{pi-region1}
\end{align}
For $(R(z))_{12}$ in (\ref{pi-region0}), we get from (\ref{r-nasy}), (\ref{p-1-residue}) and (\ref{r-1-explicit})
\begin{equation}
(R(z))_{12} = \frac{\sqrt{2}}{\sqrt{n} (z-1) y(L)} \Big( K(L) - \nu \Big) \rho^{-2} + O(n^{-1}) \quad \textrm{for } z \in
\Omega_0 \setminus B_\delta .
\end{equation}
According to our assumption $K(L) \neq \nu$ in Theorem \ref{thm3}, we know that
\begin{equation}
(R(z))_{12} \neq 0 \qquad \textrm{for } z \in \Omega_0 \setminus B_\delta \hbox{ and $n$ large enough}.
\end{equation}
Let $U(\mathcal{S})$ be the neighborhood of the Szeg\H{o} curve $\mathcal{S}$. Note that $\Gamma_{0,n}$ tends to $\mathcal{S}$
as $n \to \infty$; see Lemma \ref{gamma0-szego}. Using the similar analysis as in the proof of Proposition \ref{b-delta}, it can
be shown that there exists  $\varepsilon^* > 0 $ such that $\re \phi_{t,n}(z) > \varepsilon^*$ for $z \in \Omega_0 \setminus
U(\mathcal{S})$. This means that $e^{-2n \phi_{t,n}(z)}$ is exponentially small as $n \to \infty$. Then, from (\ref{pi-region0})
we have
\begin{equation} \label{pi-zero2}
\pi_n(z) =  \ e^{n g_{t,n}(z)} (-1)^n \frac{e^{- \nu \pi i}- e^{\nu \pi i}}{(z-1)^{2b}} \left[ \frac{\sqrt{2} \Big( K(L) - \nu
\Big)}{\sqrt{n} (z-1) y(L)} \rho^{-2} + O(n^{-1}) \right]
\end{equation}
for $z \in \Omega_0 \setminus ( B_\delta \cup U(\mathcal{S}) )$.. So, for $n$ is large enough, there is no zero of $\pi_n(z)$ in
this region.

For $z \in \Omega_1 \setminus U(\mathcal{S})$, similarly it can be shown that $e^{n \phi_{t,n}(z)}$ is exponentially small as $n
\to \infty$. Thus, we get from (\ref{pi-region1})
\begin{equation*} \label{pi-zero3}
\pi_n(z)= \left( 1 + O \left( \frac{1}{\sqrt{n}} \right) \right) e^{n g_{t,n}(z)}, \quad \hbox{$ z \in \Omega_1 \setminus
(B_\delta \cup U(\mathcal{S}))$},
\end{equation*}
which means there is no zero of $\pi_n(z)$. Thus, we can see that all zeros of $\pi_n(z)$ accumulate in $B_\delta \cup
U(\mathcal{S})$ as $n \to \infty$.

For Case II, the analysis is similar. This completes the proof of Theorem \ref{thm3}.
\end{proof}


\section{The solution to PIV with special parameters} \label{special}

It is a well-known fact that for certain parameters $\Theta$ and $\Theta_\infty$, the PIV equation has special solutions given
in terms of parabolic cylinder functions; see e.g.\ \cite{luka,murata,bch,gromak,okamoto,uw}. The special solution of PIV that
is of interest in this paper (that is, the one characterized by the Stokes multipliers (\ref{case0-sm})) turns out to be of this
type, for certain values of $b$. We will discuss it briefly in this section.

We recall the RH problem for $\Psi$ from (\ref{psi-rhb})--(\ref{psi-0}) with parameters
\begin{equation} \label{special-para}
\Theta = -b \quad \hbox{and} \quad \Theta_\infty = \nu + b,
\end{equation}
and Stokes multipliers (\ref{case0-sm}) given in terms of $b$ and $\nu$ by
\begin{equation} \label{special-sm}
\begin{array}{ll}
s_1= e^{(2 \nu + 3b) \pi i} - e^{-b \pi i}, & s_2 = e^{b \pi i}, \\
s_3 = e^{-(2 \nu + b ) \pi i} - e^{- b \pi i}, & s_4 = - e^{(2 \nu + b ) \pi i}.
\end{array}
\end{equation}
It turns out that for $b \in \mathbb Z/2$, the corresponding special solution of PIV can be expressed in terms of parabolic
cylinder functions. If $\nu \in \mathbb N$, the parabolic cylinder function reduces to a Hermite function and the PIV solution
is a rational function, see also \cite{murata,bch,okamoto,ko}.

We discuss the cases $b=0$ and $b=1/2$ here.

\subsection{The case $b=0$}

Denote by $\Psi^{(b=k/2)}$ the solution of the RH problem for $\Psi$ in Section \ref{rh-piv} with parameters $\Theta$ and
$\Theta_\infty$ given in (\ref{special-para}), Stokes multipliers given in (\ref{special-sm}) and $b = \frac{k}{2}$.

When $b = 0$, by (\ref{special-para}) and (\ref{special-sm}) we have
\begin{equation} \label{para-0}
\Theta = 0, \quad  \quad \Theta_\infty = \nu
\end{equation}
and
\begin{equation} \label{sm-0}
\begin{array}{llll}
s_1= e^{2 \nu \pi i} - 1, & s_2 = 1, & s_3 = e^{- 2 \nu \pi i} - 1, & s_4 = - e^{ 2 \nu \pi i}.
\end{array}
\end{equation}
Substituting (\ref{para-0}) and (\ref{sm-0}) into (\ref{psi-rhb})--(\ref{psi-0}), we get the RH problem for $\Psi^{(b=0)}$. For
this special case $b =0$, the RH problem is solved explicitly in terms of parabolic cylinder functions $D_\nu(\zeta)$, where
$D_\nu(\zeta)$ is the solution of the second order linear differential equation
\begin{equation} \label{pcf-eqn}
y''(\zeta) + \left(\nu - \frac{1}{2} - \frac{1}{4} \zeta^2\right) y(\zeta) = 0
\end{equation}
uniquely characterized by the asymptotic property
\begin{equation} \label{d-asy}
D_{\nu} (\zeta) = \zeta^{\nu} e^{\frac{1}{4} \zeta^2} \left( 1 - \frac{\nu (\nu -1)}{2 \zeta^2} + O \left( \frac{1}{\zeta^4}
\right) \right) \qquad \textrm{as } \zeta \to \infty, |\arg \zeta| < \frac{3 \pi}{4};
\end{equation}
see \cite[Chapter 19]{as}. Other solutions of the differential equation (\ref{pcf-eqn}) are $D_\nu(-\zeta)$ and $D_{-\nu - 1}
(\pm i \zeta)$. These solutions are related as follows
\begin{align}
    D_\nu(z) &= \frac{ \Gamma(1+\nu) }{ \sqrt{2 \pi} } \left( e^{ \frac{\nu
      }{2} \pi i } D_{-\nu-1}(i z) + e^{ - \frac{ \nu}{2} \pi i } D_{-
      \nu - 1}( -i z ) \right) \text{,} \\
    D_\nu( - z ) &= \frac{ \Gamma( 1 + \nu ) }{ \sqrt{ 2 \pi } }
      \left( e^{ - \frac{\nu}{2} \pi i } D_{-\nu-1}(i z) + e^{
      \frac{\nu}{2} \pi i } D_{- \nu - 1}( -i z ) \right)
      \text{,}  \\
    D_{\nu} ( z ) &= e^{ \nu \pi i } D_{\nu} ( -z ) -
      \frac{ \Gamma(1+\nu) }{ \sqrt{2 \pi} } (e^{\nu \pi i} -
      e^{-\nu \pi i} ) e^{ \frac{\nu}{2} \pi i } D_{-\nu-1}( -i z)
      \text{,} \\
    D_\nu(z) &= e^{-\nu \pi i} D_\nu( -z ) +
      \frac{ \Gamma(1+\nu) }{ \sqrt{2 \pi} } (e^{\nu \pi i} -
      e^{-\nu \pi i} ) e^{ -\frac{\nu}{2} \pi i } D_{-\nu-1}( i z); \label{pcf-relation4}
\end{align}
see also \cite[p.117]{bateman}. When $\nu \notin \mathbb{Z}$, the solution to the RH problem for $\Psi^{(b=0)}$ is given as
follows, which can be verified with (\ref{d-asy})--(\ref{pcf-relation4}),
\begin{equation} \label{psi-b=0}
\Psi^{(b=0)}(\l) = \begin{cases} \biggl( 2^{-\nu} e^{s^2} (e^{2\nu \pi i} - 1) \biggr)^{-\frac{\sigma_3}{2}}  V(\l)
\left(e^{2\nu \pi i} - 1 \right)^{\frac{\sigma_3}{2} }, & \re \l > 0, \\ \biggl( 2^{-\nu} e^{s^2} ( 1 - e^{- 2\nu \pi i} )
\biggr)^{-\frac{\sigma_3}{2}}  V(\l)  \left( 1 - e^{- 2\nu \pi i} \right)^{\frac{\sigma_3}{2} }, & \re \l < 0, \end{cases}
\end{equation}
where
\begin{equation} \label{pcf-sol}
V(\l) = \begin{cases} \left( \begin{matrix} e^{-\frac{\nu}{2} \pi i}
D_{-\nu}( -i \sqrt{2}(\l + s)) & -\frac{\sqrt{2 \pi} \; i}{\Gamma(\nu)} D_{\nu-1}(\sqrt{2}(\l + s)) \\
\frac{\Gamma(1+\nu)}{\sqrt{2 \pi}} e^{-\frac{\nu}{2} \pi i} D_{-\nu -1}( -i \sqrt{2}(\l + s)) & D_\nu(\sqrt{2}(\l + s))
\end{matrix} \right) \\  & \hbox{for } \arg \l \in (0, \frac{\pi}{2}), \\
\left( \begin{matrix} e^{\frac{\nu}{2} \pi i}
D_{-\nu}( -i \sqrt{2}(\l + s)) & \frac{\sqrt{2 \pi} \; i}{\Gamma(\nu)} D_{\nu-1}(-\sqrt{2}(\l + s)) \\
\frac{\Gamma(1+\nu)}{\sqrt{2 \pi}} e^{\frac{\nu}{2} \pi i} D_{-\nu -1}( -i \sqrt{2}(\l + s)) & D_\nu(-\sqrt{2}(\l + s))
\end{matrix} \right) \\ &  \hbox{for } \arg \l \in ( \frac{\pi}{2}, \pi ), \\
\left( \begin{matrix} e^{-\frac{\nu}{2} \pi i}
D_{-\nu}( i \sqrt{2}(\l + s)) & \frac{\sqrt{2 \pi} \; i}{\Gamma(\nu)} D_{\nu-1}(-\sqrt{2}(\l + s)) \\
-\frac{\Gamma(1+\nu)}{\sqrt{2 \pi}} e^{-\frac{\nu}{2} \pi i} D_{-\nu -1}( i \sqrt{2}(\l + s)) & D_\nu(-\sqrt{2}(\l + s))
\end{matrix} \right) \\ &  \hbox{for } \arg \l \in (\pi, \frac{3 \pi}{2}), \\
\left( \begin{matrix} e^{\frac{\nu}{2} \pi i}
D_{-\nu}( i \sqrt{2}(\l + s)) & -\frac{\sqrt{2 \pi} \; i}{\Gamma(\nu)} D_{\nu-1}(\sqrt{2}(\l + s)) \\
-\frac{\Gamma(1+\nu)}{\sqrt{2 \pi}} e^{\frac{\nu}{2} \pi i} D_{-\nu -1}( i \sqrt{2}(\l + s)) & D_\nu(\sqrt{2}(\l + s))
\end{matrix} \right) \\ & \hbox{for } \arg \l \in (\frac{3 \pi}{2} , 2 \pi);
\end{cases}
\end{equation}
see also \cite{dev,kmm}. Then, by (\ref{psi-b=0}) and (\ref{pcf-sol}), we have
\begin{equation}
(\Psi^{(b=0)}(\l))_{12} = - 2^{\frac{\nu}{2}} e^{- \frac{1}{2} s^2} \frac{\sqrt{2 \pi} \; i}{\Gamma(\nu) \left(e^{2\nu \pi i} -
1 \right)} D_{\nu-1}(\sqrt{2}(\l + s)) \quad \hbox{ for } \re \l >0.
\end{equation}
Recalling (\ref{psi-1}) and (\ref{d-asy}), we get from above formula
\begin{eqnarray}
y^{(b=0)}(s) & = &  2 \lim_{\l \to \infty} \l \; 2^{\frac{\nu}{2}} e^{- \frac{1}{2} s^2} \frac{\sqrt{2 \pi} \;
i}{\Gamma(\nu) \left(e^{2\nu \pi i} - 1 \right)} D_{\nu-1}(\sqrt{2}(\l + s)) e^{\frac{\l^2}{2} + s \l} \l^{-\nu} \nonumber \\
& = & c_0 \; e^{-s^2},
\end{eqnarray}
where
\begin{equation}
c_{0} =  \frac{2^{\nu + 1}  \sqrt{\pi} \, i}{\Gamma{(\nu)} \left(e^{2\nu \pi i} - 1 \right)} .
\end{equation}
Furthermore, by (\ref{y-def}) we get
\begin{equation} \label{u-b=0}
u^{(b=0)}(s) = -2s - \frac{d}{ds} \log{y^{(b=0)}(s)} \equiv 0,
\end{equation}
which is the trivial solution of the PIV equation with parameter $\Theta=0$.

It is easily seen that, with (\ref{u-b=0}),  the expressions (\ref{an-final}) and (\ref{bn-final}) for the recurrence
coefficients of the generalized Laguerre polynomials indeed reduce to (\ref{Laguerrerecurrence2}).

\subsection{The case $b= \frac{1}{2}$}

Knowing the precise form of $\Psi^{(b=0)}(\l)$, we are going to compute $\Psi^{(b=1/2)}(\l)$. For $b = \frac{1}{2}$, by
(\ref{special-para}) and (\ref{special-sm}) we have
\begin{equation}
\Theta = - \frac{1}{2}, \quad  \quad \Theta_\infty = \nu + \frac{1}{2}
\end{equation}
and
\begin{equation}
\begin{array}{llll}
s_1= -i \, e^{2\nu \pi i} + i, & s_2 = i, & s_3 = -i \, e^{-2 \nu \pi i} + i, & s_4 = - i \, e^{2 \nu \pi i}.
\end{array}
\end{equation}
Define
\begin{equation} \label{r-def}
\Phi(\l) := e^{-\frac{1}{4} \pi i \sigma_3} \Psi^{(b=1/2)}(\l) e^{\frac{1}{4} \pi i \sigma_3} (\Psi^{(b=0)}(\l))^{-1}.
\end{equation}
From the RH problems for $\Psi^{(b=0)}$ and $\Psi^{(b=1/2)}$, with $\l^{1/2}$ defined with a cut along $i \, \mathbb{R}_-$, we
easily obtain the RH problem for $\l^{1/2} \Phi(\l)$ as follows:

\begin{enumerate}

\item[(a)] $\l^{1/2} \Phi(\l)$ is analytic for $\l \in \mathbb{C} \setminus (
\mathbb{R} \cup i \, \mathbb{R} )$; see Figure~\ref{psicontour};

\item[(b)] $\big( \l^{1/2} \Phi(\l) \big)_+ = \big( \l^{1/2} \Phi(\l)
\big)_- $ for $\l \in \mathbb{R} \cup i \, \mathbb{R}$;

\item[(c)] as $\l \to \infty$
\begin{equation} \label{r-large}
\l^{1/2} \Phi(\l) = \left( \begin{matrix} 1 & -i \ (\Psi_{-1}^{(b=1/2)})_{12} \\ - (\Psi_{-1}^{(b=0)})_{21} & \l +
(\Psi_{-1}^{(b=1/2)})_{22} - (\Psi_{-1}^{(b=0)})_{22}
\end{matrix} \right) + O(\l^{-1}),
\end{equation}
where $\Psi_{-1}(s)$ is given in (\ref{psi-1});

\item[(d)] as $\l \to 0$
\begin{equation} \label{r-small}
\l^{1/2} \Phi(\l) = O(1).
\end{equation}

\end{enumerate}
From parts (b) and (d) of the RH problem it follows that $\l^{1/2} \Phi(\l)$ is an entire function in the complex $\lambda$
plane. By (\ref{r-large})  we then have for $\l \in \mathbb C$,
\begin{equation}
\l^{1/2} \Phi(\l) = \left( \begin{matrix} 1 & -i \ (\Psi_{-1}^{(b=1/2)})_{12} \\ - (\Psi_{-1}^{(b=0)})_{21} & \l
+(\Psi_{-1}^{(b=1/2)})_{22} - (\Psi_{-1}^{(b=0)})_{22}
\end{matrix} \right)
\end{equation}
Recalling $(\Psi_{-1}^{(b=1/2)})_{12} = - \frac{1}{2} \, y^{(b=1/2)}$ (see (\ref{psi-1})),  we obtain
\begin{equation} \label{psi-r}
\l^{1/2} \Phi(\l) = \left( \begin{matrix} 1 & \frac{i}{2} \, y^{(b=1/2)} \\ * & *
\end{matrix} \right),
\end{equation}
where $*$ denotes entries that are not important for what follows. By (\ref{r-def}), we have
\begin{equation}
\l^{1/2} \Psi^{(b=1/2)}(\l) =  e^{\frac{1}{4} \pi i \sigma_3} \l^{1/2} \Phi(\l) \Psi^{(b=0)}(\l) e^{-\frac{1}{4} \pi i
\sigma_3}.
\end{equation}
Combining (\ref{psi-b=0}), (\ref{pcf-sol}) and (\ref{psi-r}), we get that the $(1,2)$ entry of $\l^{1/2} \Psi^{(b=1/2)}(\l)$ for
$\l$ in the fourth quadrant is equal to
\begin{equation} \label{psi-12}
2^{\frac{\nu}{2}} e^{ - \frac{1}{2} s^2} \frac{\sqrt{2 \pi}}{\Gamma(\nu) \left(e^{2\nu \pi i} - 1 \right)} D_{\nu-1}(\sqrt{2}(\l
+ s)) - \frac{y^{(b=1/2)}(s)}{2} 2^{-\frac{\nu}{2}} e^{\frac{1}{2} s^2} D_{\nu}(\sqrt{2}(\l + s)).
\end{equation}
From (\ref{psi-0}), we know
\begin{equation}
\l^{1/2} \Psi^{(b=1/2)}(\l) = O \left( \begin{matrix} 1 & \l \\ 1 & \l \end{matrix} \right) \quad \textrm{as $\l \to 0$ and $\l$
in the fourth quadrant}.
\end{equation}
This, together with (\ref{psi-12}), gives us
\begin{equation}
2^{\frac{\nu}{2}} e^{ - \frac{1}{2} s^2} \frac{\sqrt{2 \pi}}{\Gamma(\nu) \left(e^{2\nu \pi i} - 1 \right)} D_{\nu-1}(\sqrt{2} s)
- \frac{y^{(b=1/2)}(s)}{2} 2^{-\frac{\nu}{2}} e^{\frac{1}{2} s^2} D_{\nu}(\sqrt{2} s) = 0,
\end{equation}
so that
\begin{equation}
y^{(b=1/2)}(s) = \frac{c_{1/2} \ D_{\nu-1}(\sqrt{2} \, s)}{e^{s^2} D_{\nu}(\sqrt{2} \, s)}, \qquad
 c_{1/2} = \frac{2^{\nu + 3/2}  \sqrt{\pi} }{\Gamma{(\nu)}
\left(e^{2\nu \pi i} - 1 \right)}.
\end{equation}
From (\ref{y-def}), we then see that the relevant solution of PIV in case $b=1/2$ is equal to
\begin{equation}
\begin{split}
u^{(b=1/2)}(s) & = -2s - \frac{d}{ds} \log{y^{(b=1/2)}(s)} \\
& = \frac{d}{ds} \log{ \frac{\ D_{\nu}(\sqrt{2} \, s)}{D_{\nu-1}(\sqrt{2} \, s)}}.
\end{split}
\end{equation}
From this formula, it is readily seen that the poles of $u^{(b=1/2)}(s)$ are the zeros of $D_\nu(\sqrt{2} \, s)$ and $D_{\nu
-1}(\sqrt{2} \, s)$. For $\nu$ is real, the parabolic cylinder function $D_{\nu}$ has $\max(\lceil \nu \rceil,0)$ zeros on the
real axis, where $\lceil \cdot \rceil$ is the ceiling function; see \cite[p.126]{bateman}. So, the PIV solution $u^{(b=1/2)}$
has poles on the real axis if $\nu > 0$.

\subsection{The case $b= \frac{k}{2}$, $k \in \mathbb Z$}

Repeating the previous calculations, it is possible (at least in principle) to find the solution $u^{(b = k/2)}(s)$ for PIV with
$ k \in \mathbb{Z}$ in terms of parabolic cylinder functions. For example, for $b=1$ we have
\begin{equation}
y^{(b=1)}(s) = c_1 \frac{\mathcal{W}\Big( D_{\nu-1}(\sqrt{2} \, s), D_{\nu}(\sqrt{2} \, s) \Big)}{e^{s^2} \; \mathcal{W}\Big(
D_{\nu}(\sqrt{2} \, s), D_{\nu+1}(\sqrt{2} \, s) \Big)},
\end{equation}
where $\mathcal{W}$ is the Wronskian with respect to $s$ and
\begin{equation}
c_{1} = - \frac{2^{\nu + 2}  \sqrt{\pi} \, i}{\Gamma{(\nu)} \left(e^{2\nu \pi i} - 1 \right)} .
\end{equation}
Thus,
\begin{equation}
u^{(b=1)}(s) = \frac{d}{ds} \log{ \frac{\mathcal{W}\Big( D_{\nu}(\sqrt{2} \, s), D_{\nu+1}(\sqrt{2} \, s)
\Big)}{\mathcal{W}\Big( D_{\nu-1}(\sqrt{2} \, s), D_{\nu}(\sqrt{2} \, s) \Big)} }.
\end{equation}
Again we see that there are poles on the real axis if $\nu > 0$.

\medskip

\textbf{Acknowledgement.} We would like to thank Professor Peter Clarkson for helpful discussions and comments.


\end{document}